\documentclass[10pt]{amsart}
\usepackage{amsmath, latexsym, amsthm, amsfonts,bm,amssymb,mathscinet} % Useful packages.
\usepackage[colorlinks,citecolor=blue,urlcolor=blue,linkcolor=magenta]{hyperref}
\usepackage{appendix}
\usepackage{natbib} % Package for the bibliography.
\usepackage{url} % Package for embedding url.
\usepackage{enumerate}
\usepackage{mathtools} % mathclap  
\usepackage{color} % package added 27 april 2014

\bibpunct{(}{)}{;}{a}{}{,} % Punctuation for citations.

\theoremstyle{plain}
\newtheorem{thm}{Theorem}
\newtheorem{prop}{Proposition}
\newtheorem{lemma}{Lemma}

\newtheorem{assump}{Assumption}

\theoremstyle{remark}
\newtheorem{rem}{Remark}

\def\ex{{\rm {\mathbb E\,}}}

\newcommand{\pp}{\mathbb{P}}
\newcommand{\ee}{\mathbb{E}}

\renewcommand{\th}{\theta}
\newcommand{\la}{\lambda}

\newcommand{\bb}[1]{\mathbb{ #1}}

\newcommand{\dd}{\mathrm{d}}
\newcommand{\eps}{\varepsilon}

\newcommand{\one}{\mathbf{1}}
\newcommand{\half}{\frac{1}{2}}

\allowdisplaybreaks

\usepackage[OT2,OT1]{fontenc}
\newcommand\cyr{%
\renewcommand\rmdefault{wncyr}%
\renewcommand\sfdefault{wncyss}%
\renewcommand\encodingdefault{OT2}%
\normalfont
\selectfont}
\DeclareTextFontCommand{\textcyr}{\cyr}

\newcommand{\si}{\sigma}

\newcommand{\ba}{\mathbf{a}}

\begin{document}

\title[A non-parametric Bayesian approach to decompounding]{A non-parametric Bayesian approach to decompounding from high frequency data}

\author{Shota Gugushvili}
\address{Mathematical Institute\\
Leiden University\\
P.O. Box 9512\\
2300 RA Leiden\\
The Netherlands}
\email{shota.gugushvili@math.leidenuniv.nl}

\author{Frank van der Meulen}
\address{Delft Institute of Applied Mathematics\\
Faculty of Electrical Engineering, Mathematics and Computer Science\\
Delft University of Technology\\
Mekelweg 4\\
2628 CD Delft\\
The Netherlands}
\email{f.h.vandermeulen@tudelft.nl}

\author{Peter Spreij}
\address{Korteweg-de Vries Institute for Mathematics\\
University of Amsterdam\\
P.O.\ Box 94248\\
1090 GE Amsterdam\\
The Netherlands}
\email{spreij@uva.nl}

\thanks{The research leading to these results has received funding from the European Research Council under ERC Grant Agreement 320637.}

\subjclass[2000]{Primary: 62G20, Secondary: 62M30}

\keywords{Compound Poisson process; Non-parametric Bayesian estimation; Posterior contraction rate; high frequency observations.}

\begin{abstract}

Given a sample from a discretely observed compound Poisson process, we consider non-parametric estimation of the density $f_0$ of its jump sizes, as well as of its intensity $\lambda_0.$ We take a Bayesian approach to the problem and specify the prior on $f_0$ as the Dirichlet location mixture of normal densities. An independent prior for $\lambda_0$ is assumed to be compactly supported and to possess a positive density with respect to the Lebesgue measure. We show that under suitable assumptions the posterior contracts around the pair $(\lambda_0,f_0)$ at essentially (up to a logarithmic factor) the $\sqrt{n\Delta}$-rate, where $n$ is the number of observations and $\Delta$ is the mesh size at which the process is sampled. The emphasis is on high frequency data, $\Delta\to 0$, but the obtained results are also valid for fixed $\Delta$. In either case we assume that $n\Delta\rightarrow\infty$. Our main result implies existence of Bayesian point estimates converging (in the frequentist sense, in probability) to $(\lambda_0,f_0)$ at the same rate. 

We also discuss a practical implementation of our approach. The computational problem is dealt with by inclusion of auxiliary variables and we develop a Markov Chain Monte Carlo algorithm that samples from the joint distribution of the unknown parameters in the mixture density and the introduced auxiliary variables. Numerical examples illustrate the feasibility of this approach. 
\end{abstract}

\date{\today}

\maketitle

\section{Introduction}
\label{intro}

\subsection{Problem formulation and announcement of the main result}\label{section:pf}
Let $N=(N_{t},\, t\geq 0)$ be a Poisson process with a constant intensity $\lambda>0$ and let $Y_1,Y_2,Y_3\ldots$ be a sequence of independent random variables independent of $N$ and having a common distribution function $F$ with density $f$ (with respect to the Lebesgue measure).  A compound Poisson process (abbreviated CPP) $X=(X_t,\, t\geq 0)$ is defined as
\begin{equation}
\label{cpp}
X_t=\sum_{j=1}^{N_t} Y_j,
\end{equation}
where the sum over an empty set is by definition equal to zero. CPPs form a basic model in a variety of applied fields, most notably in e.g.\ queueing and risk theory, see~\cite{embrechts97} and~\cite{prabhu98} and the references therein,  but also in other fields of science, see e.g.\  \cite{Alexandersson1985}, \cite{br1993} for stochastic models for precipitation, \cite{katz} on modelling of hurricane damage, or \cite{scalas} for applications in economics and finance.

Suppose that corresponding to the `true' parameter values $\lambda=\lambda_0$ and $f=f_0,$ a discrete time sample  $X_{\Delta},X_{2\Delta},\ldots,X_{n\Delta}$ is available from \eqref{cpp}, where $\Delta>0.$ Such a discrete time observation scheme is common in a number of applications of CPP, e.g.\ in the precipitation models of the above references.
Based on the sample  $\mathcal{X}_n^\Delta=(X_\Delta,X_{2\Delta},\ldots,X_{n\Delta}),$ we are interested in (non-parametric) estimation of $\lambda_0$ and $f_0.$ Before proceeding further, we notice that by the stationary independent increments property of a compound Poisson process, the random variables $Z_i^{\Delta} = X_{i\Delta}-X_{(i-1)\Delta}$, $1\le i \le n$, are independent and identically distributed. Each $Z_i^{\Delta}$ has the same distribution as the random variable
\begin{equation}
\label{psum}
Z^{\Delta}=\sum_{j=1}^{T^{\Delta}} Y_j,
\end{equation}
where $T^{\Delta}$ is independent of the sequence $Y_1, Y_2,\ldots$ and has a Poisson distribution with parameter $\Delta\lambda$. Hence, our problem is equivalent to estimating (non-parametrically) $\lambda_0$ and $f_0$ based on the sample $\mathcal{Z}_n^{\Delta}=(Z_1^{\Delta},Z_2^{\Delta},\ldots,Z_n^{\Delta})$. 
We will henceforth use this alternative formulation of the problem. Our emphasis is on \emph{high frequency} data, $\Delta=\Delta_n\to 0$ as $n\rightarrow\infty,$ but the obtained results are also valid for \emph{low frequency} observations, i.e.\ for fixed $\Delta$. 

Our main result is on the contraction rate of the posterior distribution, which we show to be, up to a logarithmic factor, $(n\Delta)^{-1/2}$. A by now standard approach to obtain contraction rates in an IID setting is to verify the assumptions of the fundamental Theorem~2.1 in \cite{ghosal00}. It should be noted that in the present high frequency setting, this theorem is not applicable. One of the model assumptions underlying this theorem, which is satisfied in  \cite{gugu15}, is that one deals with samples of a \emph{fixed} distribution, whereas in our present high frequency observation regime the distribution of $Z^\Delta$ is \emph{varying}, with the Dirac distribution concentrated at zero as its limit for $\Delta\to 0$. Therefore we propose an alternative approach, circumventing the use of the cited Theorem~2.1. The theoretical contribution of the present paper is therefore not only the statement of the main result itself, but also its proof. Next to this we also discuss a practical implementation of our non-parametric Bayesian approach, a Markov Chain Monte Carlo algorithm that samples from the joint distribution of the unknown parameters in the mixture density and certain introduced auxiliary variables. 

\subsection{Literature review and present approach}

Because adding a Poisson number of $Y_j$'s amounts to compounding their distributions, the problem of recovering the intensity $\lambda_0$ and the density $f_0$ from the observations $Z_i$'s can be referred to as decompounding. Decompounding already has some history: the early contributions~\cite{buchmann03} and~\cite{buchmann04} dealt with estimation of the distribution function $F_0,$ paying particular attention to the case when $F_0$ is discrete, while the later contributions~\cite{comte13},~\cite{duval12} and~\cite{vanes07} concentrated on estimation of the density $f_0$ instead. More (frequentist) theory on statistical inference on compound Poisson processes (and more generally on L\'evy processes) can be found in the volume~\cite{levymatters}, with the survey paper~\cite{cgc}  devoted to statistical methods for  high frequency discrete observations, with a special section on compound Poisson processes. Other references on statistics for L\'evy processes in the high frequency data setting are \cite{comte11}, \cite{cgc2010a}, \cite{cgc2010b}, \cite{fl2008}, \cite{fl2009}, and \cite{uk2011}.
All these approaches are frequentist in nature. On the other hand, theoretical and computational advances made over the recent years have shown that a non-parametric Bayesian approach is feasible in various statistical settings; see e.g.\ \cite{hjort10} for an overview. This is the approach we will take in this work to estimate $\lambda_0$ and $f_0.$ 

To the best of our knowledge, non-parametric Bayesian approach to inference for (a class of) L\'evy processes was first considered in \cite{gugu15}. That paper, contrary to the present context, dealt with observations at fixed equidistant times, and was strongly based on an application of Theorem~2.1 of \cite{ghosal00}, as already alluded to in the Problem formulation of Section~\ref{section:pf}. 
The present work complements the results from \cite{gugu15}, in the sense that we now allow high frequency observations, which requires a  substantially different route to prove our results, as we will explain in more detail in Section~\ref{int_res}. 

We will study the non-parametric Bayesian approach to decompounding from a frequentist point of view (in the sense specified below), so that one may also think of it as a means for obtaining a frequentist estimator. Advantages of the non-parametric Bayesian approach include automatic quantification of uncertainty in parameter estimates through Bayesian posterior credible sets and automatic selection of the degree of smoothing required in non-parametric inferential procedures.

%-------

\subsection{Results}
\label{int_res}
The non-parametric class $\mathcal{F}$ of densities $f$ that we consider is that of location mixtures of normal densities. So we consider densities specified by
\begin{equation}\label{eq:fh}
f(x)=f_{ H , \sigma } (x)=\int \phi_{\sigma }(x-z)\mathrm{d}H(z),
\end{equation}
where $\phi_{\sigma }$ denotes the density of the normal distribution with mean zero and variance $\sigma^2$ and $H$ is a mixing measure.
These mixtures form a rich and flexible class of densities, see \cite{mw1992} and \cite{mp2000}, that are capable of closely approximating many densities that themselves are not representable in this way. The resulting mixture densities will be infinitely smooth, which is arguably the case in many, if not most, practical applications.

Bayesian estimation requires specification of prior distributions on $\lambda$ and $f.$ We propose independent priors  on $\lambda$ and $f$ that we denote by  $\Pi_1$ and $\Pi_2,$ respectively. For $f$, we take a   Dirichlet mixture of normal densities as a prior.  This type of prior in the context of Bayesian density estimation has  been introduced in \cite{ferguson83} and \cite{lo84}; for recent references see e.g.\ \cite{ghosal01}. The prior for $f$ is defined as  the law of the function
$f_{ H , \sigma }$ as in \eqref{eq:fh}, with $H$ assumed to follow a Dirichlet process prior $D_\alpha$ with base measure $\alpha$ and $\sigma$ a-priori independent with distribution  $\Pi_3$. Recall that a Dirichlet process $D_{\alpha}$ on $\mathbb{R}$ with the base measure $\alpha$ defined on the Borel $\sigma$-algebra $\mathcal{B}(\mathbb{R})$ (we assume $\alpha$ to be non-negative and $\sigma$-additive) is a random probability measure $G$ on $\mathbb{R},$ such that for every finite and measurable partition $B_1,B_2,\ldots,B_k$ of $\mathbb{R},$ the probability vector $(G(B_1),G(B_2),\ldots,G(B_k))$ possesses the Dirichlet distribution on the $k$-dimensional simplex with parameters $(\alpha(B_1),\alpha(B_2),\ldots\alpha(B_k)).$ See e.g.\ the original paper~\cite{ferguson73}, or the overview article~\cite{ghosal10} for more information on Dirichlet process priors.  
%
%Location mixtures of  normals form a rich class of densities that are capable of closely approximating many densities that themselves are not representable as such mixtures. 

A nonparametric Bayesian approach to density estimation employing a Dirichlet mixture of normal densities as a prior can in very rough sense be thought of as a Bayesian counterpart of kernel density estimation (with a Gaussian kernel), cf.\ \cite{vdv07}, p.~697.

With the sample size $n$ tending to infinity, the Bayesian approach should be able to discern the true parameter pair $(\lambda_0,f_0)$ with increasing accuracy. We can formalise this by requiring, for instance, that for any fixed neighbourhood $A$ (in an appropriate topology) of $(\lambda_0,f_0),$ $\Pi(A^c|\mathcal{Z}_{n}^\Delta)\rightarrow 0$ in $\mathbb{Q}_{\lambda_0,f_0}^{\Delta, n}$-probability. %Here $\Pi= \Pi_1\times \Pi_2$ 
Here $\Pi$ is used as a shorthand notation for the posterior distribution of $(\lambda,f)$
and we use  $\mathbb{Q}_{\lambda_0,f_0}^{\Delta}$ to denote the law of the random variable $Z^{\Delta}$ in \eqref{psum} and  $\mathbb{Q}_{\lambda_0,f_0}^{\Delta, n}$ the law of $\mathcal{Z}_n^{\Delta}$. More generally, one may take a sequence of shrinking neighbourhoods $A_n$ of $(\lambda_0,f_0)$ and try to determine the rate at which the neighbourhoods $A_n$ are allowed to shrink, while still capturing most of the posterior mass. This rate is referred to as a posterior convergence rate (we will give the precise definition in Section~\ref{results}). Two fundamental references dealing with establishing it in various statistical settings are~\cite{ghosal00} and~\cite{ghosal01}. This convergence rate can be thought of as an analogue of the convergence rate of a frequentist estimator. The analogy can be made precise: contraction of the posterior distribution at a certain rate implies existence of a Bayes point estimate with the same convergence rate (in the frequentist sense); see Theorem 2.5 in~\cite{ghosal00} and the discussion on pp. 506--507 there.

%Obviously, for our programme to be successful, $\Delta$ has to satisfy certain assumptions. We will assume that $\Delta=n^{-\alpha},$ where $0\leq\alpha<1,$ which covers both the case of the so called low frequency observation (fixed $\Delta$ independent of $n$) and the high frequency observation schemes ($\Delta\rightarrow 0$). With this assumption we will in any case have $n\Delta\rightarrow\infty,$ which is a necessary condition for consistent estimation of $(\lambda_0,f_0),$ as it ensures that asymptotically we observe an infinite number of jumps in the process.

Obviously, for our programme to be successful, $\Delta$ has to satisfy the assumption $n\Delta\rightarrow\infty,$ which is a necessary condition for consistent estimation of $(\lambda_0,f_0),$ as it ensures that asymptotically we observe an infinite number of jumps in the process. We cover both the case of so called high frequency observation schemes ($\Delta\rightarrow 0$) as well as low frequency observations (fixed $\Delta$). A sufficient condition, which covers both observation regimes and which relates $\Delta$ to $n$,  is  $\Delta=n^{-\alpha}$, where $0\leq\alpha<1$.

We note that in~\cite{ghosal06} and~\cite{tang07} non-parametric Bayesian inference for Markov processes is studied, of which compound Poisson processes form a particular class, but these papers deal with estimation of the transition density of a discretely observed Markov process, which is   different from the problem  we consider here. A parametric Bayesian approach to inference for compound Poisson processes is studied in \cite{insua12}, Sections 5.5 and 10.3.

The main result of our paper is Theorem~\ref{mainthm}, in which we state sufficient conditions on the prior that yield a posterior rate of contraction of the order $(\log^{\kappa}(n\Delta))/ \sqrt{n\Delta},$ for some constant $\kappa>0.$  We argue that this rate is a nearly (up to a logarithmic factor) optimal posterior contraction rate in our problem. Our main result complements the one in \cite{gugu15}, in that it treats both the low and high frequency observation schemes simultaneously, with emphasis on the latter. We note (again) a fundamental difference between the present paper and \cite{gugu15}, when it comes down to the techniques to prove the main result. As Theorem~2.1 of \cite{ghosal00} cannot immediately be used, we take an alternative tour that avoids this theorem, but instead refines a number of technical results involving properties of statistical tests that form essential ingredients of the proof in \cite{ghosal00}. These refined results are then used as key technical steps in a direct proof of our Theorem~\ref{mainthm}.
Furthermore, it establishes the posterior contraction rate for infinitely smooth jump size densities $f_0,$ which is not covered by \cite{gugu15}. On the other hand, \cite{gugu15} deals with multi-dimensional CPPs, while in this paper we consider only the one-dimensional case. Finally, 
%unlike \cite{gugu15}, 
in this work we also discuss a practical implementation of our non-parametric Bayesian approach. The computational problem is dealt with by inclusion of auxiliary variables. More precisely, we show how a Markov Chain Monte Carlo algorithm can be devised that samples from the joint distribution of the unknown parameters in the mixture density and the introduced auxiliary variables. Numerical examples illustrate the feasibility of this approach. 

%------------

\subsection{Organisation}
The remainder of the paper is organised as follows. In the next section we state some preliminaries on the likelihood, prior and notation. 
%In section~\ref{scaled} we motivate the use of the scaled Hellinger metric to define neighbourhoods for which posterior contraction rate is derived in case the observations are sampled at high frequency.
%In Section~\ref{results} we present the main result on posterior contraction rate (Theorem~\ref{mainthm}), whose proof is given in Section~\ref{sec:proofmainthm}.   We discuss the numerical implementation of our results in Section~\ref{section:algo}. Technical lemmas and their proofs used to prove the main theorem are gathered in the appendices.
In Section~\ref{results} we first motivate the use of the scaled Hellinger metric to define neighbourhoods for which posterior contraction rate is derived in case the observations are sampled at high frequency.
Then we present the main result on the posterior contraction rate (Theorem~\ref{mainthm}), whose proof is given in Section~\ref{sec:proofmainthm}.   We discuss the numerical implementation of our results in Section~\ref{section:algo}. Technical lemmas and their proofs used to prove the main theorem are gathered in the Appendix.

% \subsubsection{Probability measures}

%%%%%%%%%%%%%
\section{Preliminaries and notation}

\subsection{Likelihood, prior and posterior}
We are interested in Bayesian inference with Bayes' formula. Therefore we need to specify the likelihood in our model. We use the following notation:
\smallskip
\begin{center}
\begin{tabular}{|l|l|}
\hline
$\mathbb{P}_f$ & law of $Y_1$  (law of the jumps of the CPP) \\
$\mathbb{Q}_{\lambda, f}^{\Delta}$ & law of $Z_1^{\Delta}$ (law of the increments of the discretely observed CPP) \\
$\mathbb{Q}_{\lambda,f}^{\Delta, n}$ &  law of $\mathcal{Z}_n^{\Delta}$ (joint law of the increments of the discretely observed CPP)\\
$\mathbb{R}_{\lambda, f}^{\Delta}$ & law of $(X_t,\, t\in [0,\Delta])$ (law of the CPP on $[0,\Delta]$)\\
\hline
\end{tabular}
\end{center}

\smallskip
The characteristic function of the Poisson sum $Z^{\Delta}$ defined in \eqref{psum} is given by
\begin{equation*}
\phi(t)=e^{-\lambda\Delta+\lambda\Delta \phi_{f}(t)},
\end{equation*}
where $\phi_{f}$ is the characteristic function of $f.$ This can be rewritten as
\begin{equation*}
\phi(t)=e^{-\lambda\Delta}+(1-e^{-\lambda\Delta})\frac{1}{e^{\lambda\Delta}-1}\left( e^{\lambda\Delta \phi_{f}(t)} -1 \right),
\end{equation*}
which, using the fact that $\phi_{f}$ vanishes at infinity, shows that the distribution of $Z^{\Delta}$ is a mixture of a point mass at zero and an absolutely continuous distribution. Letting $t\rightarrow\infty,$ we get that $\phi(t)\rightarrow e^{-\lambda\Delta}.$ Hence $\lambda$ is identifiable from the law of $Z^{\Delta}$, and then so is $f$. The density of the law $\mathbb{Q}_{\lambda, f}^{\Delta}$ of $Z^{\Delta}$ with respect to the measure $\mu$, which is the sum of Lebesgue measure and the Dirac measure concentrated at zero, can in fact be written explicitely as (cf.\ p.~681 in~\cite{vanes07} and Proposition 2.1 in~\cite{duval12})
\begin{equation}
\label{density}
\frac{\dd \mathbb{Q}_{\lambda, f}^{\Delta}}{\dd\mu}(x)=e^{-\lambda\Delta}\one_{\{0\}}(x)+(1-e^{-\lambda\Delta})\sum_{m=1}^{\infty} a_m(\lambda\Delta) f^{\ast m}(x)\one_{\mathbb{R}\setminus\{0\}}(x),
\end{equation}
where $\one_A$ denotes the indicator of a set $A$, 
\begin{equation}
\label{am}
a_m(\lambda\Delta)=\frac{1}{ e^{\lambda\Delta} -1 } \frac{(\lambda\Delta)^m}{m!},
\end{equation}
and $f^{\ast m}$ denotes the $m$-fold convolution of $f$ with itself. However, the expression \eqref{density} is useless for Bayesian computations. To work around this problem, we will employ a different dominating measure. Consider the law $\mathbb{R}_{\lambda,f}^{\Delta}$ of $(X_t, t\in [0,\Delta])$. 
%Hence we need to find another dominating measure for $\mathbb{Q}_{\lambda,f}$ in our setup. A natural idea is to take as the dominating measure the law of the compound Poisson sum for some fixed parameter values $(\widetilde{\lambda},\widetilde{f}).$ 
By the Theorem on p.~261 in~\cite{skorohod64}, $\mathbb{R}_{\lambda,f}^{\Delta}$ is absolutely continuous with respect to $\mathbb{R}_{\widetilde{\lambda},\widetilde{f}}^{\Delta}$ if and only if $\mathbb{P}_{f}$ is absolutely continuous with respect to $\mathbb{P}_{\widetilde{f}}$ (we of course assume that $\lambda,\widetilde{\lambda}>0$). %We shall need something more than this, namely the equivalences $\mathbb{P}_{\lambda,f}  \sim \mathbb{P}_{\widetilde{\lambda},\widetilde{f}}$ and $\mathbb{P}_{f}  \sim \mathbb{P}_{\widetilde{f}}$ (this is needed in Section~\ref{results} in order to define certain divergences between these probability measures and also in the proofs). 
A simple condition to ensure the latter is to assume that $\widetilde{f}$ is continuous and does not take the value zero on $\mathbb{R}.$ 
%A specific choice of $\widetilde{\lambda}$ and $\widetilde{f}$ is unimportant for our purposes: we will only assume that $0<\widetilde{\lambda},\widetilde{f}\leq 1.$\footnote{waar gebruiken we $\widetilde{f}\leq 1.$?}

Define the random measure $\mu$ by
\begin{equation*}
\mu ( B ) = \{ \# t :(t,X_t-X_{t-})\in B \}, \quad B \in \mathcal{B}([0,\Delta])\otimes\mathcal{B}(\mathbb{R}\setminus\{0\}).
\end{equation*}
Under $\mathbb{R}_{{\lambda},{f}},$ the random measure $\mu$ is a Poisson point process on $[0,\Delta]\times(\mathbb{R}\setminus\{0\})$ with intensity measure $ {\Lambda}(dt,dx)={\lambda} dt {f}(x)dx,$ which follows e.g.\ from Theorem 1 on p.~69 and Corollary on p.~64 in~\cite{skorohod64}. By formula (46.1) on p.~262 in~\cite{skorohod64}, we have
\begin{equation}\label{eq:contlik}
\frac{\mathrm{d} \mathbb{R}_{\lambda,f}^{\Delta}}{ \mathrm{d}\mathbb{R}_{\widetilde{\lambda},\widetilde{f}}^{\Delta} }(X)=\exp\left( \int_0^{\Delta} \int_{\mathbb{R}} \log\left( \frac{\lambda f(x)}{{\widetilde{\lambda}} \widetilde{f}(x)} \right) \mu(dt,dx) - \Delta(\lambda-{\widetilde{\lambda}}) \right).
\end{equation}
 By Theorem 2 on p.~245 in~\cite{skorohod64} and Corollary 2 on p.~246 there, the density $k_{\lambda,f}^{\Delta}$ of $\mathbb{Q}_{\lambda,f}^{\Delta}$ with respect to $\mathbb{Q}_{\widetilde{\lambda},\widetilde{f}}^{\Delta}$ is given by the conditional expectation
\begin{equation}
\label{eq:k}
k_{\lambda,f}^{\Delta}(x)=\ex_{\widetilde{\lambda},\widetilde{f}}\left( \frac{\mathrm{d} \mathbb{R}_{\lambda,f}^{\Delta}}{ \mathrm{d}\mathbb{R}_{\widetilde{\lambda},\widetilde{f}}^{\Delta} }(X)\, \middle|\, X_{\Delta}=x \right),
\end{equation}
where the subscript in the conditional expectation operator signifies the fact that it is evaluated under the probability $\mathbb{R}_{ \widetilde{\lambda},\widetilde{f} }^{\Delta}.$
Hence the likelihood (in the parameter pair $(\lambda,f)$) associated with the sample $\mathcal{Z}_n^{\Delta}$ is given by the product
\begin{equation}
\label{ln}
L_n^{\Delta}(\lambda,f)=\prod_{i=1}^n k_{\lambda,f}^{\Delta}(Z_i^{\Delta}).
\end{equation}
An advantage of specifying the likelihood in this manner is that it allows one to reduce some of the difficult computations for the laws $\mathbb{Q}_{\lambda,f}^{\Delta}$ to those for the laws $\mathbb{R}_{\lambda,f}^{\Delta},$ which are simpler.

%
% Hence our prior $\Pi$ on $f$ is a Dirichlet mixture of normals, see e.g.\~\cite{ghosal01} for additional information and references on such a prior. With probability one, realisations of a Dirichlet process prior are discrete measures (see e.g.\~\cite{ferguson73}, p.\ 218), so that realisations of Dirichlet location mixtures of the normals are finite or countable sums of the type $\sum_j \phi_{\sigma}(x-z_j) p_j $ for probability vectors $\{p_j\}$ (note in particular that realisations of Dirichlet location mixtures of the normals are continuous and strictly positive). 

Observe that the priors on $\lambda$ and $f$ indirectly induce the prior  $\Pi= \Pi_1\times \Pi_2$ on the collection of densities $k_{\lambda,f}^{\Delta}.$ We will indiscriminately use the symbol $\Pi$ to signify both the prior on $(\lambda,f),$ but also on the density $k_{\lambda,f}^{\Delta}.$ The posterior in the first case will be understood as the posterior for the pair $(\lambda,f),$ while in the second case as the posterior for the density $k_{\lambda,f}^{\Delta}.$ 
We will often use the same symbol $\Pi$ to denote the posterior distribution of $(\lambda,f)$ and on the density $k_{\lambda,f}^{\Delta}$.
This simplifies notationally some of the formulations below.

%Let
%\begin{equation*}
%\mathcal{F}=\{ f_{H,\sigma}: H\in\mathcal{M}({B}), \underline{\sigma}\leq\sigma\leq\overline{\sigma} \},
%\end{equation*}
%where by definition $\mathcal{M}(B)$ is the collection of all probability distributions on the set $B\subset\mathbb{R}.$

By Bayes' theorem, the posterior measure of any measurable set 
%$A\subset[\underline{\lambda},\overline{\lambda}]\times\mathcal{F}$ 
$A\subset(0,\infty)\times\mathcal{F}$ 
is given by
\begin{equation*}
\Pi(A|\mathcal{Z}_n^{\Delta})=\frac{ \iint_A L_n^{\Delta}(\lambda,f) \dd \Pi_1(\lambda) \mathrm{d}\Pi_2(f) }{ \iint L_n^{\Delta}(\lambda,f) \dd\Pi_1(\lambda) \mathrm{d}\Pi_2(f) }.
\end{equation*}
Upon setting $\overline{A}=\{ k_{\lambda,f}:(k,\lambda)\in A \}$  and recalling our conventions above, this can also be written as
\begin{equation*}
\Pi(\overline{A}|\mathcal{Z}_n^{\Delta})=\frac{ \int_{\overline A} L_n^{\Delta}(k) \mathrm{d}\Pi(k) }{ \int L_n^{\Delta}(k) \mathrm{d}\Pi(k) }.
\end{equation*}
Once the posterior is available, one can next proceed with computation of other quantities of interest in Bayesian statistics, such as Bayes point estimates or credible sets.

\subsection{Notation}
Throughout the paper we will use the following notation to compare two sequences $\{a_n\}$ and $\{b_n\}$ of positive real numbers: $a_n\lesssim b_n$ will mean that there exists a constant $C>0$ that is independent of $n$ and is such that $a_n\leq C b_n,$ while $a_n\gtrsim b_n$ will signify the fact that $a_n\geq C b_n.$ 
%$a_n\gtrsim b_n$ will mean that there exists a constant $C>0$ that is independent of $n$ and is such that $C a_n\geq b_n$. 
%$a_n\asymp b_n$ will mean that $a_n$ and $b_n$ are asymptotically of the same order, i.e.\ $a_n\lesssim b_n$ and $a_n\gtrsim b_n$.
%\ $-\infty<\liminf_{n\rightarrow\infty} a_n/b_n\leq \limsup_{n\rightarrow\infty} a_n/b_n<\infty.$

%\subsection{Distances}
Next we introduce various notions of distances between probability measures.
The Hellinger distance $h(\mathbb{Q}_{0},\mathbb{Q}_{1})$ between two probability laws $\mathbb{Q}_{0}$ and $\mathbb{Q}_{1}$ on a measurable space $(\Omega,\mathfrak{F})$ is defined as
\[
h(\mathbb{Q}_{0},\mathbb{Q}_{1})=\left( \int \left( \mathrm{d}\mathbb{Q}_{0}^{1/2} - \mathrm{d}\mathbb{Q}_{1}^{1/2} \right)^2 \right)^{1/2}.
\]
Assume further $\mathbb{Q}_0\ll\mathbb{Q}_1$. The Kullback-Leibler (or informational) divergence $\mathrm{K}(\mathbb{Q}_{0},\mathbb{Q}_{1})$  is defined as
\[\mathrm{K}(\mathbb{Q}_{0},\mathbb{Q}_{1})= \int \log \left(\frac{ \mathrm{d}\mathbb{Q}_{0}}{ \mathrm{d}\mathbb{Q}_{1} } \right) \mathrm{d}\mathbb{Q}_{0},
\]
while the $\mathrm{V}$-discrepancy is defined through
\[	\mathrm{V}(\mathbb{Q}_{0},\mathbb{Q}_{1})
=\int \log^2 \left(\frac{ \mathrm{d}\mathbb{Q}_{0}}{ \mathrm{d}\mathbb{Q}_{1} } \right) \mathrm{d}\mathbb{Q}_{0}.
\]
Here is some additional notation. For $f,g$ nonnegative integrable functions, not necessarily densities, we write
\begin{align*}
h^2(f,g) & =\int (\sqrt{f}-\sqrt{g})^2, \\
\mathrm{K}(f,g) & = \int \log\frac{f}{g}\,f -\int f + \int g\\
\mathrm{V}(f,g) & = \int \log^2\frac{f}{g}\,f.
\end{align*}
Note that these `distances' are all nonnegative and only zero if $f=g$ a.e. If $f$ and $g$ are densities of probability measures $\mathbb{Q}_0$ and $\mathbb{Q}_1$ on $(\mathbb{R},\mathcal{B})$ respectively, then the above `distances' reduce to the previously introduced ones.

We will also use $\mathrm{K}(x,y)=x\log\frac{x}{y}-x+y$ for $x,y>0$. Note that also $\mathrm{K}(x,y)\geq 0$ and $\mathrm{K}(x,y)=0$ if and only if $x=y$.

%\section{Scaled distances}\label{scaled}

%%%%%%%%%%%%%%%%%%%%%%%%%
\section{Main result on posterior contraction rate}
\label{results}

Denote the true parameter values for the compound Poisson process by $(\lambda_0, f_0)$. 
Recall that the problem is to estimate $f_0$ and $\lambda_0$ based on the observations $\mathcal{Z}^\Delta_n$ and that $\Delta\to 0$ in a high frequency regime. To say that a pair $(f,\lambda)$ lies in a neighbourhood of $(f_0,\lambda_0)$, one needs a notion of distance on the corresponding measures $\mathbb{Q}^{\Delta}_{\lambda,f}$ and $\mathbb{Q}^{\Delta}_{\lambda_0,f_0}$, the two possible induced laws of $Z^\Delta_i=X_{i\Delta}-X_{(i-1)\Delta}$. The Hellinger distance is a popular and rather reasonable choice to that end in non-parametric Bayesian statistics. However, for $\Delta\to 0$ the Hellinger metric $h$ between those laws automatically tends to 0. The first assertion of Lemma~\ref{lemma:delta} below states that $h(\mathbb{Q}^{\Delta}_{\lambda,f},\mathbb{Q}^{\Delta}_{\lambda_0,f_0})$ is of order $\sqrt{\Delta}$ when $\Delta\to 0$. This motivates to replace the ordinary Hellinger metric $h$ with the scaled metric $h^\Delta=h/\sqrt{\Delta}$ in our asymptotic analysis for high frequency data.   Of course, for fixed $\Delta$ (in which case one can take $\Delta=1$ w.l.o.g.), nothing changes with this replacement. The lemma also shows that the Kullback-Leibler divergence and the V-discrepancy are of order $\Delta$ for $\Delta\to 0$. Therefore we will also use the scaled distances $\mathrm{K}^\Delta=\mathrm{K}/\Delta$ and $\mathrm{V}^\Delta=\mathrm{V}/\Delta$ 

\begin{lemma}\label{lemma:delta}
The following expressions hold true: \label{eq:kdelta}
\begin{align}
\lim_{\Delta\to 0}\frac{1}{\Delta}
h^2(\mathbb{Q}^\Delta_{\lambda,f},\mathbb{Q}^\Delta_{\lambda_0,f_0}) & =
h^2(\lambda f,\lambda_0 f_0)=\int(\sqrt{\lambda f(x)}-\sqrt{\lambda_0f_0(x)})^2\,\dd x 
\label{eq:hhdelta}, \\
\lim_{\Delta\to 0}\frac{1}{\Delta}\mathrm{K}(\mathbb{Q}^\Delta_{\lambda,f},\mathbb{Q}^\Delta_{\lambda_0,f_0}) & = \mathrm{K}(\lambda f,\lambda_0 f_0)=\lambda \mathrm{K}(f,f_0)+\mathrm{K}(\lambda,\lambda_0) \label{eq:kkdelta}, \\
\lim_{\Delta\to 0}\frac{1}{\Delta}\mathrm{V}(\mathbb{Q}^\Delta_{\lambda,f},\mathbb{Q}^\Delta_{\lambda_0,f_0}) & = \mathrm{V}(\lambda f,\lambda_0 f_0)=\int \log^2\frac{\lambda f(x)}{\lambda_0 f_0(x)}\,\lambda f(x)\,\dd x. \label{eq:vvdelta}
\end{align}
\end{lemma}
The proof will be presented in Appendix~\ref{app:proof}. 
%{\color{red} The proof is incomplete. We do not use (10) and (11), right? Het leek me niet zo nodig om drie keer vrijwel hetzelfde bewijs op te schrijven. Die andere twee gebruiken voor de $B^\Delta$ bollen.}

\begin{rem}
The Hellinger process (here deterministic) of order $\half$ for \emph{continuous} observations of $X$ on an interval $[0,t]$ is given by~\cite[Sections IV.3 and IV.4a]{js} 
\[
h_t=\frac{t}{2}\int(\sqrt{\lambda f(x)}-\sqrt{\lambda_0f_0(x)})^2\,\dd x=
 h_1t,
\]
 from which it follows that $h^2(\mathbb{R}^t_{\lambda,f},\mathbb{R}^t_{\lambda_0,f_0})=2-2\exp(-h_t)$, whose derivative in $t=0$ is the same as in
\eqref{eq:hhdelta} and thus equal to $2h_1$. 
For the Kullback-Leibler divergence and the discrepancy $\mathrm{V}$ similar assertions hold. These observations have the following heuristic explanation. For $\Delta\to 0$, there is no big difference between observing the path of $X$ over the interval $[0,\Delta]$ and $X_\Delta$, as the probability of $\{N_\Delta\geq 2\}$ is small (of order $\Delta^2$). 
\end{rem}
In order to determine the posterior contraction rate in our problem, we now specify suitable neighbourhoods $A_n$ of $(\lambda_0,f_0),$ for which this will be done.
Let $M>0$ be a constant and let $\{\varepsilon_n\}$ be a sequence of positive numbers, such that  $\varepsilon_n\rightarrow 0$ as $n\rightarrow\infty.$ Let
\begin{equation*}
h^\Delta(\mathbb{Q}_0,\mathbb{Q}_1)=\frac{1}{\sqrt{\Delta}}h(\mathbb{Q}_0,\mathbb{Q}_1)
\end{equation*}
be a rescaled Hellinger distance. Lemma~\ref{lemma:delta} suggests that this is the right scaling to use. Introduce the complements of the Hellinger-type neighbourhoods of $(\lambda_0,f_0),$
\begin{equation*}
A(\varepsilon_n,M)=\{ (\lambda,f) :  h^\Delta( \mathbb{Q}_{\lambda_0,f_0}^{\Delta},\mathbb{Q}_{\lambda,f}^{\Delta} )> M \varepsilon_n \}.
\end{equation*}
We shall say that $\varepsilon_n$ is a posterior contraction rate, if there exists a constant $M>0,$ such that
\begin{equation}
\label{post_cons}
\Pi(A(\varepsilon_n,M)|\mathcal{Z}_n^{\Delta})\rightarrow 0
\end{equation}
in $\mathbb{Q}_{\lambda_0,f_0}^{\Delta, n}$-probability as $n\rightarrow\infty$. 
Our goal in this section is to determine the `fastest' rate at which $\varepsilon_n$ is allowed to tend to zero, while not violating \eqref{post_cons}.

We
will assume that the observations are generated from a compound Poisson process that satisfies the following assumption.
\begin{assump}\label{ass:truth}

\begin{enumerate}[(i)]
\item $\lambda_0$ is in a compact set $[\underline{\lambda}, \overline{\lambda}]\subset(0,\infty)$; 
\item 
The true density $f_0$ is a location mixture of normal densities, i.e.\
\begin{equation*}
f_0(x)=f_{ H_0 , \sigma_0 } (x)=\int \phi_{\sigma_0 }(x-z)\mathrm{d}H_0(z)
\end{equation*}
for some fixed distribution $H_0$ and a constant $\sigma_0 \in [\underline{\sigma}, \overline{\sigma}]\subset(0,\infty)$. Furthermore, for some $0<\kappa_0<\infty,$ $H_0[-\kappa_0,\kappa_0]=1,$ i.e.\ $H_0$ has  compact support.
\end{enumerate}
\end{assump}
The more general location-scale mixtures of normal densities,
\begin{equation*}
f_0(x)=f_{ H_0 , K_0 } (x)=\iint \phi_{\sigma}(x-z)\mathrm{d}H_0(z)\mathrm{d}K_0(\sigma),
\end{equation*}
possess even better approximation properties than the location mixtures of the normals (here $H_0$ and $K_0$ are distributions) and could also be considered in our setup. However, this would lead to additional technical complications, which could obscure essential contributions of our work.

For obtaining posterior contraction rates we need to make some assumptions on the prior. 
\begin{assump}\label{ass:prior}
\mbox{}
\begin{enumerate}[(i)]
\item The prior on $\lambda$, 
$\Pi_1$,  has a density $\pi_1$ (with respect to the Lebesgue measure) that is supported on the finite interval $[\underline{\lambda},\overline{\lambda}]\subset (0,\infty)$ and is such that
\begin{equation}
\label{pi1}
0<\underline{\pi}_1 \leq \pi_1(\lambda)\leq \overline{\pi}_1<\infty, \quad \lambda\in[\underline{\lambda},\overline{\lambda}]
\end{equation}
for some constants $\underline{\pi}_1$ and $\overline{\pi}_1$;

\item The base measure $\alpha$ of the Dirichlet process prior $D_{\alpha}$ has a continuous density on an interval $[-\kappa_0-\zeta,\kappa_0+\zeta]$, with $\kappa_0$ as in Assumption~\ref{ass:truth}~(ii), for some $\zeta>0,$ is bounded away from zero there, and for all $t>0$ satisfies the tail condition
\begin{equation}
\label{tailc}
\alpha(|z|>t)\lesssim e^{-b|t|^{\delta}}
\end{equation}
with some constants $b>0$ and $\delta>0$;

\item The prior on $\sigma$, $\Pi_3$, is supported on the interval $[\underline{\sigma},\overline{\sigma}]\subset(0,\infty)$ and is such that its density $\pi_3$ with respect to the Lebesgue measure satisfies
\begin{equation*}
0<\underline{\pi}_3 \leq \pi_3(\sigma)\leq \overline{\pi}_3<\infty, \quad \sigma\in[\underline{\sigma},\overline{\sigma}]
\end{equation*}
for some constants $\underline{\pi}_3$ and $\overline{\pi}_3.$

\end{enumerate}
\end{assump}
Assumptions~\ref{ass:truth} and~\ref{ass:prior} parallel those given in \cite{ghosal01} in the context of non-parametric Bayesian density estimation using the Dirichlet location mixture of normal densities as a prior. We refer to that paper for an additional discussion.

The following is our main result. Note that it covers both the case of high frequency observations ($\Delta\to 0$) and observations with  fixed intersampling intervals.  We use $\Pi$ to denote the posterior on $(\lambda,f)$.

\begin{thm}
\label{mainthm}
Under Assumptions~\ref{ass:truth} and~\ref{ass:prior}, provided 
%$\Delta=n^{-\alpha}$ for $0\leq\alpha<1,$
$n\Delta\to\infty$,
there exists a constant $M>0,$ such that for
\begin{equation*}
\eps_n=\frac{\log^{\kappa}(n\Delta)}{\sqrt{n\Delta}}, \quad \kappa=\max\left(\frac{2}{\delta},\frac{1}{2}\right)+\frac{1}{2},
\end{equation*}
we have
\begin{equation*}
\Pi\left( A\left(\eps_n,M\right) \middle| \mathcal{Z}_n^{\Delta}\right)\rightarrow 0
\end{equation*}
%or equivalently,
%\begin{equation*}
%\Pi\left( \overline{A}\left(\frac{\log^{\kappa}(n)}{\sqrt{n}},M\right) \middle| \mathcal{Z}_n\right)\rightarrow 0
%\end{equation*}
in $\mathbb{Q}_{\lambda_0,f_0}^{\Delta, n}$-probability as $n\rightarrow\infty.$
\end{thm}

For fixed $\Delta$ (w.l.o.g.\ one may then assume $\Delta=1$) the posterior contraction rate in Theorem~\ref{mainthm} reduces to $\eps_n=\frac{\log^{\kappa}(n)}{\sqrt{n}}$.
We also see that the posterior contraction rate  is controlled by the parameter $\delta$ of the tail behaviour in \eqref{tailc}. Note that if \eqref{tailc} is satisfied for some $\delta>4,$ it is also automatically satisfied for all $0<\delta\leq 4$.
The stronger the decay rate in \eqref{tailc}, the better the contraction rate, but all $\delta\geq 4$ give the same value $\kappa=1$. The best possible posterior contraction rate in Theorem~\ref{mainthm} for minimal $\delta$ is obtained for $\delta=4$. In the proof in Section~\ref{sec:proofmainthm} we can therefore assume that $\delta\leq 4$.
  
As on p.~1239 in \cite{ghosal01} and similar Corollary 5.1 there, Theorem~\ref{mainthm} implies existence of a point estimate of $(\lambda_0,f_0)$ with a frequentist convergence rate $\varepsilon_n.$ The (frequentist) minimax convergence rate for estimation of $k_{\lambda,f}^{\Delta}$ relative to the Hellinger distance is unknown in our problem, but an analogy to~\cite{ibragimov82} suggests that up to a logarithmic factor it should be of order $\sqrt{n\Delta}$ (cf.\ \cite{ghosal01}, p.\ 1236). The logarithmic factor is insignificant for all practical purposes. The convergence rate of an estimator of the L\'evy density with loss measured in the $L_2$-metric in a more general L\'evy model  than the CPP model is $(n\Delta)^{-\beta/(2\beta+1)},$ whenever the target density is Sobolev smooth of order $\beta$ (cf.\ \cite{comte11}). Our contraction rate is hence, roughly speaking, a limiting case of the convergence in \cite{comte11} for $\beta\to\infty.$

%%%%%%%%%%%%%%%%
\section{Algorithms for drawing from the posterior}\label{section:algo}

In this section we discuss computational methods for drawing from the distribution of the pair $(\la, f)$, conditional on $\scr{X}_n^\Delta$ (or equivalently: conditional on $\scr{Z}_n^\Delta$).  In the following there is no specific need that the observational times are equidistant.
We will assume observations at times $0<t_1<\cdots<t_n$ and set $\Delta_i =t_i-t_{i-1}$ ($1\le i \le n$). Further, for consistency with notation following shortly, we set $z_i=X_{t_i}-X_{t_{i-1}}$ and $z=(z_1,\ldots, z_n)$. We will use ``Bayesian notation'' throughout and write $p$ for a probability density of mass function and use $\pi$ similarly for a prior density or mass function. 

In general, it is infeasible to generate independent realisations of the posterior distribution of  $(\la,f)$.  To see this: from \eqref{density} one obtains that the conditional density of a nonzero increment $z$ on a time interval of length $\Delta$ is given by
\begin{equation}
\label{eq:z} 
p(z\mid \la, f) = \frac{e^{-\la\Delta}}{1-e^{-\la\Delta}}  \sum_{k=1}^\infty \frac{(\la\Delta)^k}{k!} f^{*k}(z), \end{equation}
which generally is rather intractable due to the infinite weighted sum of convolutions. We specialise to the case where the jump size distribution is a mixture of $J\ge 1$ Gaussians. The richness and versatility of the class of finite normal mixtures is convincingly demonstrated in \cite{mw1992}.

Hence, we assume
\begin{equation}
\label{eq:mixture} f(\cdot)=\sum_{j=1}^J \rho_j \phi(\cdot; \mu_j, 1/\tau), \quad \sum_{j=1}^J \rho_j=1,
\end{equation}
where $\phi(\cdot; \mu, \si^2)$ denotes the density of a random variable with $\mathcal{N}(\mu, \si^2)$ distribution. Note that in \eqref{eq:mixture} we parametrise the density with the \emph{precision} $\tau$. 
In the ``simple'' case $J=2$ the convolution density of $k$ independent jumps is given by 
\[  f^{*k}(\cdot) = \sum_{\ell=0}^k {k \choose \ell} \rho_1^\ell \rho_2^{k-\ell} \phi(\cdot; \ell\mu_1 +(k-\ell) \mu_2; k/\tau). \] 
Plugging this expression into equation \eqref{eq:z} confirms the intractable form of $p(z\mid \la, f)$. 

We will introduce auxiliary variables to circumvent the intractable form of the likelihood.  In case the CPP is observed {\it continuously}, the problem is much easier as now the continuous time likelihood on an interval $[0,T]$ is known to be (\cite{shreve}, Theorem 11.6.7)
\begin{equation*}\label{eq:likelihood}	\la^{|V|} e^{-\la T} \prod_{i \in V} f(J_i), \end{equation*}
where the $T_i$ are the jump times of the CPP, $J_i$ the corresponding jump sizes and $V=\{i:\, T_i \le T\}$. 
The tractability of the continuous time likelihood naturally suggests the construction of a data augmentation scheme. Denote the values of the CPP in between times $t_{i-1}$ and $t_i$ by $x_{(i-1,i)}$. We will refer to $x_{(i-1,i)}$ as the missing values on the $i$-th segment. Set 
\[ x^{mis} = \{x_{(i-1,i)},\: 1\le i \le n\}.\]
A data augmentation scheme now consists of augmenting auxiliary variables $x^{mis}$ to $(\la,f)$ and constructing a Markov chain that has $p(x^{mis}, \la, f \mid z)$ as invariant distribution. More specifically, a standard implementation of this algorithm consists of the following steps:
\begin{enumerate}
\item[1.] Initialise $x^{mis}$.
\item[2.] Draw $(\la, f) \mid (x^{mis}, z)$.
\item[3.] Draw $x^{mis} \mid (\la, f, z)$. 
\item[4.] Repeat steps 2 and 3 many times.
\end{enumerate}
Under weak conditions, the iterates for $(\la,f)$ are (dependent) draws from the posterior distribution. 
Step 3 entails generating compound Poisson bridges. By the Markov property, bridges on different segments can be drawn independently. Data augmentation has been used in many Bayesian computational problems, see e.g.\ \cite{tanner-wong}.  The outlined scheme can be applied to the problem at hand, but we explain shortly that imputation of complete CPP-bridges (which is nontrivial) is unnecessary and we can do with less imputation, thereby effectively reducing the state space of the Markov chain.  

As we  assume that the jumps are drawn from a non-atomic distribution, imputation is only necessary on segments with nonzero increments. For this reason we let 
\[	\scr{I} = \{i \in \{1,\ldots, n\}:\: z_i \neq 0\}  \]
 denote the set of observations with nonzero jump sizes and define the number of segments with nonzero jumps to be $I =|\scr{I}|$. 

%%%%%%%
\subsection{Auxiliary variables}

 Note that if $Y \sim f$ with $f$ as in \eqref{eq:mixture}, then $Y$ can be simulated by first drawing its label $L$, which equals $j$ with probability $\rho_j$, and next drawing from the $N(\mu_{L}, 1/\tau)$ distribution. Knowing the labels, sampling the jumps conditional on their sum being $z$ is much easier compared to the case with unknown labels.  Adding auxiliary variables as labels is  a standard trick used for inference in mixture models (see e.g.\ \cite{diebolt-robert}, \cite{richardsen-green}. For the problem at hand, we can do with even less imputation:  all we need to know is the number of jumps of each type on every segment with nonzero jump size. 
For $i \in \scr{I}$ and $j \in \{1,\ldots, J\}$, let $n_{ij}$ denote the number of jumps of type $j$ on segment $i$.  Denote the set of all auxiliary variables by $\ba=\{a_i, \: i \in I\}$, where
\[ a_i=(n_{i1},n_{i2}, \ldots, n_{iJ}). \]
In the following we will use the following additional notation: for $i=1,\ldots,n$, $j=1,\ldots, J$ we set
\[ n_i = \sum_{j=1}^J n_{ij} \qquad s_j = \sum_{i=1}^n n_{ij} \qquad s= \sum_{j=1}^J s_j\]
These are the number of jumps on the $i$-th segment, the total number of jumps of type $j$ (summed over all segments) and the total number of jumps of all types respectively.

%----
\subsection{Reparametrisation and prior specification}
Instead of parametrising with $(\la, \rho_1,\ldots, \rho_J)$, we define
\[\psi_j = \la \rho_j, \qquad j=1,\ldots, J.\] Then 
\[\la = \sum_{j=1}^J \psi_j,\quad \rho_j=\frac{\psi_j}{\sum_{j=1}^J \psi_j}.\] 
The background of this reparametrisation is the obervation that a compound Poisson random variable $Z$ whose jumps are of $J$ types can be decomposed as $Z=\sum_{j=1}^JZ_j$, where the $Z_j$ are independent, compound Poisson random variables whose jumps are of type $j$ only, and where the parameter of the Poisson random variable is $\psi_j$. In what follows we use $\th=(\psi,\mu, \tau)$ with $\psi=(\psi_1,\ldots, \psi_J)$ and $\mu=(\mu_1, \ldots, \mu_J)$.

Denote the Gamma distribution with shape parameter $\alpha$ and rate $\beta$ by $\mathcal{G}(\alpha,\beta)$. We take priors
\begin{align*}
\psi_1,\ldots, \psi_J \quad &\stackrel{\text{iid}}{\sim} \quad \mathcal{G}(\alpha_0, \beta_0) \\ 
\mu \mid \tau \quad &\sim \quad   \mathcal{N}([\xi_1,\ldots, \xi_J]', I_{J\times J}(\tau\kappa)^{-1}) \\
 \tau \quad & \sim \quad \mathcal{G}(\alpha_1,\beta_1)
\end{align*}
with positive hyperparameters $(\alpha_0, \beta_0, \alpha_1, \beta_1,  \kappa)$ fixed. 

%----
\subsection{Hierarchical model and data augmentation scheme}
We construct a Metropolis-Hastings algorithm to draw from 
\[ p(\th, \ba \mid z)  =\frac{p(\th, z, \ba)}{p(z)}. \]
%where 
%\[	\th=(\psi, \mu, \tau), \qquad \psi=(\psi_1,\ldots, \psi_J),\qquad  . 	\]
For an index $i\in\scr{I}$ we set $\ba_{-i} = \{ a_j,\: j \in \scr{I}\setminus \{i\}\}$. The two main steps of the algorithm are:
\begin{enumerate}
\item {\it Update segments:} for each segment $i \in \scr{I}$, draw $a_i$ conditional on $(\th, z, \ba_{-i})$;
\item {\it Update parameters:} draw $\th$ conditional on $(z, \ba)$. 
\end{enumerate}
Compared to the full data augmentation scheme discussed previously, the present approach is computationally much cheaper as the amount of imputation scales with the number of segments that need imputation.
 If the time in between observations is fixed and equal to $\Delta$, then the expected number of segments for imputation equals $n \left(1- e^{-\la \Delta}\right)$, which is for small $\Delta$ approximately proportional to $n\Delta\lambda$. 

%For simulating  a realisation of the CPP, it is useful to write our model as a hierarchical model
%\begin{align}\label{eq:hierar}
% z_i \mid a_i, \mu, \tau  \quad\stackrel{\text{ind}}{\sim}  \quad & N(a_i'\mu, n_i/\tau)  \nonumber\\
%(n_{i1},\ldots, n_{iJ}) \mid n_i, \psi \quad \stackrel{\text{ind}}{\sim}  \quad & \begin{cases}  \text{MultNom}(n_i, \psi_1/\la, \ldots, \psi_J/\la)  &\qquad \text{if} \quad n_i>0 \\  (0,\ldots, 0) & \qquad \text{if} \quad n_i=0 \end{cases} \\
%n_i \mid \psi  \quad \stackrel{\text{ind}}{\sim} \quad & \text{Pois}(\la \Delta_i)\nonumber \\
%(\psi, \mu, \tau)  \quad \sim \quad & \pi(\psi, \mu, \tau).\nonumber
%\end{align}
%Here $i \in \{1,\ldots, n\}$. 
%As a result, we obtain 
%\[ p(\th, z, \ba)= \pi(\th) \times \prod_{i=1}^n e^{-\la \Delta_i} \frac{(\la\Delta_i)^{n_i}}{n_i!} \times
%\prod_{i \in \scr{I}} \left(n_i! \prod_{j=1}^J \frac{(\psi_j/\la)^{n_{ij}}}{n_{ij}!} \right) \times \left(\prod_{i=1}^n  \phi(z_i ; a_i' \mu , n_i/\tau)\right).
%\]
%This expression can be simplified to 
%\[ p(\th, z, \ba)= \pi(\th) \times \prod_{j=1}^J e^{-\psi_j T}\times  \prod_{i=1}^n\left( \prod_{j=1}^J \frac{(\psi_j\Delta_i)^{n_{ij}}}{n_{ij}!} \phi(z_i ; a_i' \mu , n_i/\tau)\right).\]

%Define the Normal distribution with mean $\mu$ and variance $0$ to be the Dirac measure at $\mu$. 
Denote the Poisson distribution with mean $\la$ by $\mathcal{P}(\la)$.  Including the auxiliary variables, we can write the observation model as a {\it hierarchical model}
\begin{align}\label{eq:hierar2}
 z_i \mid a_i, \mu, \tau  \quad \stackrel{\text{ind}}{\sim}  \quad &  N(a_i'\mu, n_i/\tau)\nonumber \\
n_{ij} \mid \psi \quad  \stackrel{\text{ind}}{\sim} \quad &  \mathcal{P}(\psi_j\Delta_i) \\
(\psi, \mu, \tau)  \quad \sim \quad & \pi(\psi, \mu, \tau)\nonumber
\end{align}
(with  $i\in \{1,\ldots, n\}$ and $j\in \{1,\ldots, J\}$). This implies
\[ p(\th, z, \ba)= \pi(\th) \times   \prod_{i=1}^n\left( \phi(z_i ; a_i' \mu , n_i/\tau) \prod_{j=1}^J e^{-\psi_j\Delta_i} \frac{(\psi_j\Delta_i)^{n_{ij}}}{n_{ij}!}\right).\]

%---

%-----
\subsection{Updating  segments}\label{subsec:updating-segments}
Updating the $i$-th segment requires drawing from 
\[	p(a_i \mid \th, z, \ba_{-i}) \propto \phi(z_i ; a_i' \mu , n_i/\tau) \prod_{j=1}^J \frac{(\psi_j\Delta_i)^{n_{ij}}}{n_{ij}!}. 
\]
We do this with a Metropolis-Hastings step.   First we draw a proposal $n_i^\circ$ (for $n_i$) from a $\mathcal{P}(\la \Delta_i$) distribution, conditioned to have nonzero outcome.  Next, we draw 
\[	a_i^\circ= (n_{i1}^\circ,\ldots, n^\circ_{iJ}) \sim \mathcal{MN}(n_i^\circ; \psi_1/\la,\ldots, \psi_J/\la), \] 
where $\mathcal{MN}$ denotes the multinomial distribution. 
Hence the proposal density equals
\begin{align*} q(n_{i1}^\circ,\ldots, n^\circ_{iJ} \mid \th) &= \frac{e^{-\la \Delta_i}}{1-e^{-\la \Delta_i}} \frac{(\la \Delta_i)^{n_i^\circ}}{n_i^\circ !} {n^\circ_i \choose n^\circ_{i1} \cdots n^\circ_{iJ}}
\prod_{j=1}^J (\psi_j/\la)^{n^\circ_{ij}}\\
& =\frac{e^{-\la \Delta_i}}{1-e^{-\la \Delta_i}}  \prod_{j=1}^J \frac{(\psi_j\Delta_i)^{n_{ij}^\circ}}{n^\circ_{ij}!}.
\end{align*}
The acceptance probability for the proposal $n^\circ$ equals $1\wedge A$, with 
\[ A= \frac{\phi(z_i ; (a_i^\circ)' \mu , n^\circ_i/\tau)}{\phi(z_i ; a_i' \mu , n_i/\tau) }. \]
%Different segments are updated independently, allowing for parallelisation. Variants of this proposal mechanism can be used equally easy. One may for example first generate 
%\[ \rho^\circ | \rho  \sim \mbox{Dirichlet}(1+\nu \rho_1,\ldots, 1+\nu \rho_J) \]
%followed by generating $a_i^\circ$ from a  $\mbox{MultNom}(n_i^\circ; \rho^\circ_1,\ldots, \rho^\circ_J)$ distribution. Computation of the corresponding acceptance probability is straightforward. The parameter $\nu$ is a tuning parameter that essentially controls the variance of the distribution of $\rho^\circ$ conditional on $\rho$. 

\subsection{Updating parameters}\label{sec:upd_pars}
The proof of the following lemma is given in Appendix~\ref{proof2}. 
\begin{lemma}
\label{lem:postpars}
Conditional on $\ba$, $\psi_1,\ldots \psi_J$ are independent and 
	\[ \psi_j \mid \ba \sim \mathcal{G}(\alpha_0+s_j, \beta_0+T).\]  
Furthermore, 
\begin{equation}\label{eq:update-mutau}
\begin{split}
\mu \mid \tau, z, \ba &\sim \mathcal{N}\left(P^{-1}q, \tau^{-1} P^{-1}\right), \\
\tau \mid z, \ba &\sim \mathcal{G}(\alpha_1+I/2, \beta_1+(R-q' P^{-1} q)/2)),
\end{split}
\end{equation}
where $P$ is the symmetric $J\times J$ matrix with elements
\begin{equation}
\label{eq:P} P= \kappa I_{J\times J} + \tilde{P} \qquad \tilde{P}_{j,k} = \sum_{i\in \scr{I}} n_i^{-1}n_{ij} n_{ik}, \quad j,k \in \{1,\ldots, J\},
\end{equation}
$q$ is the $J$-dimensional vector with 
\begin{equation}
\label{eq:q}
 q_j=\kappa \xi_j +\sum_{i \in \scr{I}} n_i^{-1} n_{ij} z_i,
\end{equation}
$R>0$ is given by 
\begin{equation}
\label{eq:R}
 R= \kappa \sum_{j=1}^J \xi_j^2 + \sum_{i \in \scr{I}} n_i^{-1} z_i^2, 
\end{equation}
and $R-q' P^{-1} q >0$.
\end{lemma}

\begin{rem}
If for some $j\in \{1,\ldots, J\}$ we have $s_j=0$ (no jumps of type $j$), then the matrix $\tilde{P}$ is singular. However, adding $\kappa I_{J\times J}$ ensures invertibility of $P$. 
\end{rem}

\subsection{Numerical illustrations}

The first two examples concern mixtures of two normal distributions
We simulated $n=5.000$ segments with $\Delta=1$,  $\mu_1=2$, $\mu_2=-1$ and $\tau=1$. For the prior-hyperparameters we took $\alpha_0=\beta_0=\alpha_1=\beta_1=1$, $\xi_1=\xi_2=0$ and $\kappa=1$. 

The results for $\lambda\Delta=1$, $\rho_1=0.8$, $\rho_2=0.2$ and hence $\psi_1=0.8$ and  $\psi_2=0.2$ are shown in Figure~\ref{fig:la=1}. The densities obtained from the posterior mean of the parameter estimates and the true density are shown in Figure~\ref{fig:la=1_dens}. The average acceptance probability for updating the segments was $51\%$. 

The results for $\lambda\Delta=3$, $\rho_1=0.8$, $\rho_2=0.2$ and hence $\psi_1=2.4$ and $\psi_2=0.6$ are shown in Figure~\ref{fig:la=3}. The densities obtained from the posterior mean of the parameter estimates and the true density are shown in Figure~\ref{fig:la=3_dens}. The average acceptance probability for updating the segments was $41\%$. Observe that the autocorrelation functions of the iterations of the $\psi_i$ in the second case display a much slower decay.

%\medskip

We also assessed the performance of our method on a more complicated example where we took a mixture of four normals. Here $\Delta=1$, $(\mu_1, \mu_2, \mu_3, \mu_4)=(-1, 0, 0.8, 2)$, $(\psi_1,\psi_2, \psi_3, \psi_4)=(0.3, 0.4, 0.2, 0.1)$ (hence $\la=1$) and $\tau^{-1}=0.09$. The results obtained after simulating $n=10.000$ segments are  shown in Figures~\ref{fig:multimodal-trace} and~\ref{fig:multimodal-density}. 

Mixtures of normals need not be multimodal and can also yield skew densities. As an example, we consider the case where $(\mu_1,\mu_2)=(0,2)$, $(\psi_1, \psi_2)=(1.5,0.5)$ (hence $\la=2$) and $\tau=1$. Data were generated and discretely sampled with $\Delta=1$ and $n=5.000$ segments. A plot of the posterior mean is shown in Figure~\ref{fig:skew-density}. 

\begin{figure}
\includegraphics[width=0.95\textwidth]{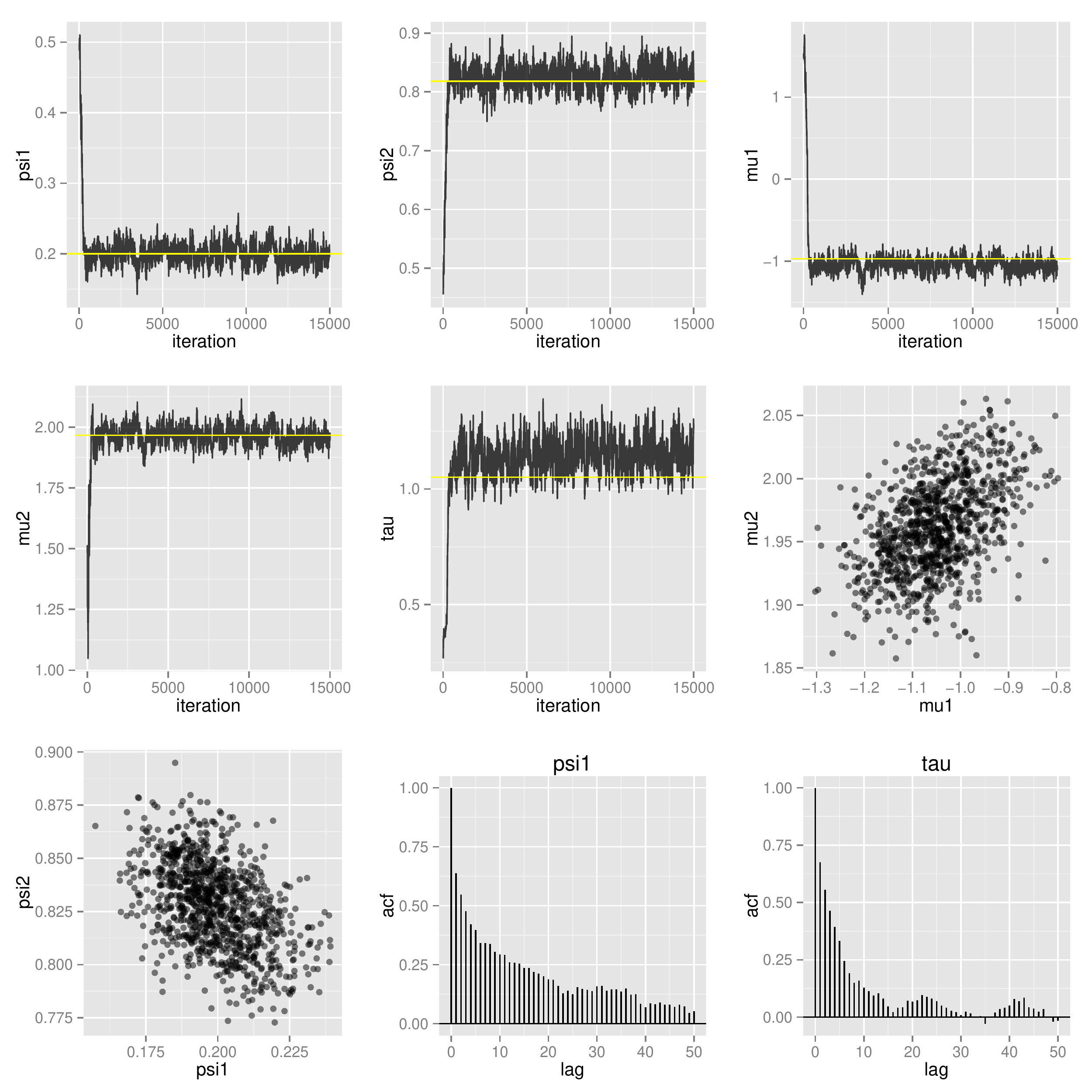}
\caption{Results for $\la=1$ using 15.000 MCMC iterations. 
The trace plots show all iterations; in the other plots the first 5.000 iterations are treated as burnin. 
The figures are obtained after subsampling the iterates, where only each 5th iterate was saved.
The horizontal yellow lines are obtained from computing the posterior mean  of $\th$ based on the true auxiliary variables on all segments.}
\label{fig:la=1}
\end{figure}

\begin{figure}
\includegraphics[scale=0.6]{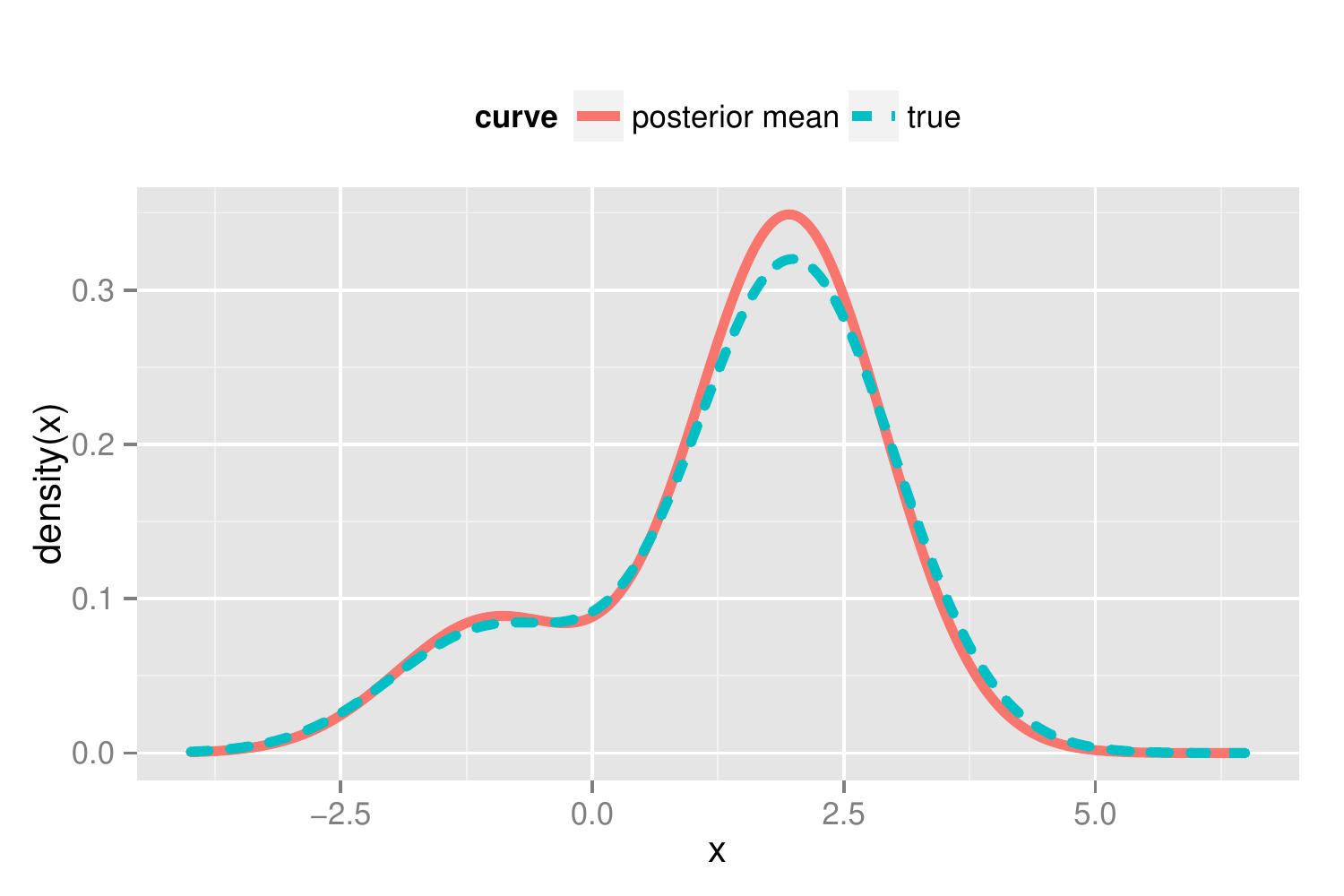}
\caption{Results for $\la=1$;  the first 5.000 iterations are treated as burnin. Shown are the true jump size density and the density obtained from the posterior mean of the non-burnin iterates.}\label{fig:la=1_dens}
\end{figure}

%------

\begin{figure}
\includegraphics[width=0.95\textwidth]{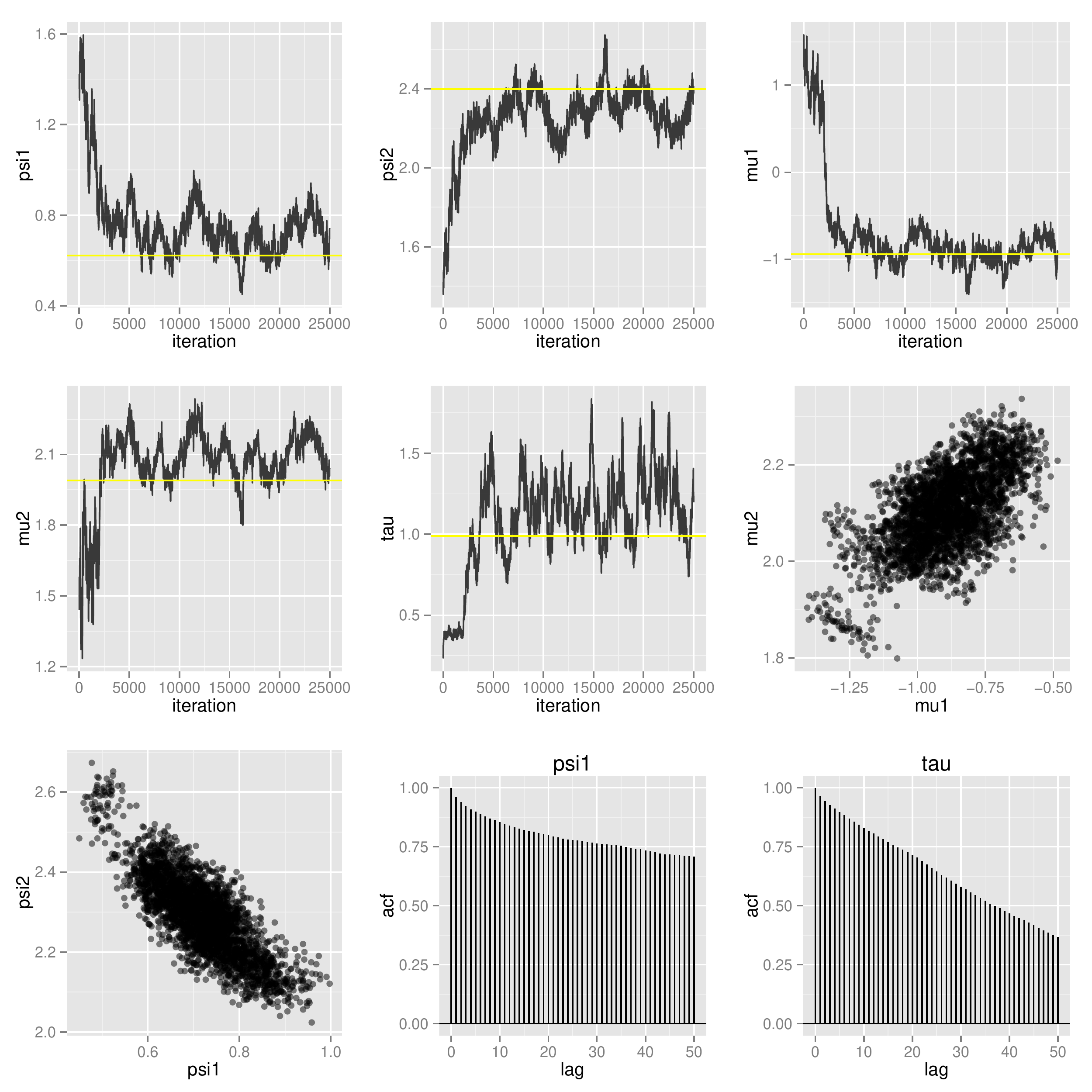}
\caption{Results for $\la=3$ using 25.000 MCMC iterations. 
The trace plots show all iterations; in the other plots the first 10.000 iterations are treated as burnin. 
The figures are  obtained after subsampling the iterates, where only each 5th iterate was saved.
The horizontal yellow lines are obtained from computing the posterior mean  of $\th$ based on the true auxiliary variables on all segments.}
\label{fig:la=3}
\end{figure}

\begin{figure}
\includegraphics[scale=0.6]{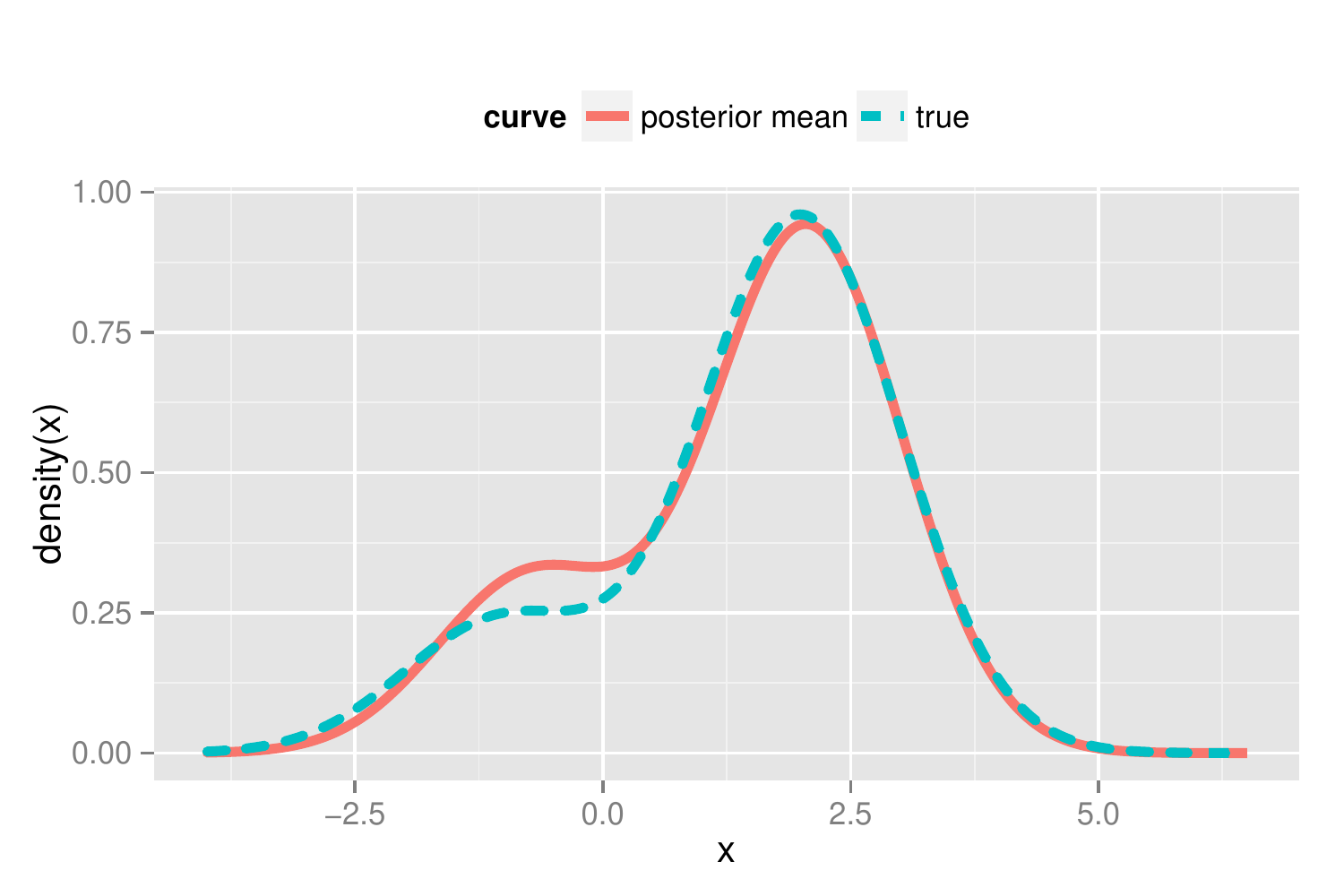}
\caption{Results for $\la=3$;  the first 10.000 iterations are treated as burnin. Shown are the true jump size density and the density obtained from the posterior mean of the non-burnin iterates.
}\label{fig:la=3_dens}
\end{figure}

%------

\begin{figure}
\includegraphics[width=0.95\textwidth]{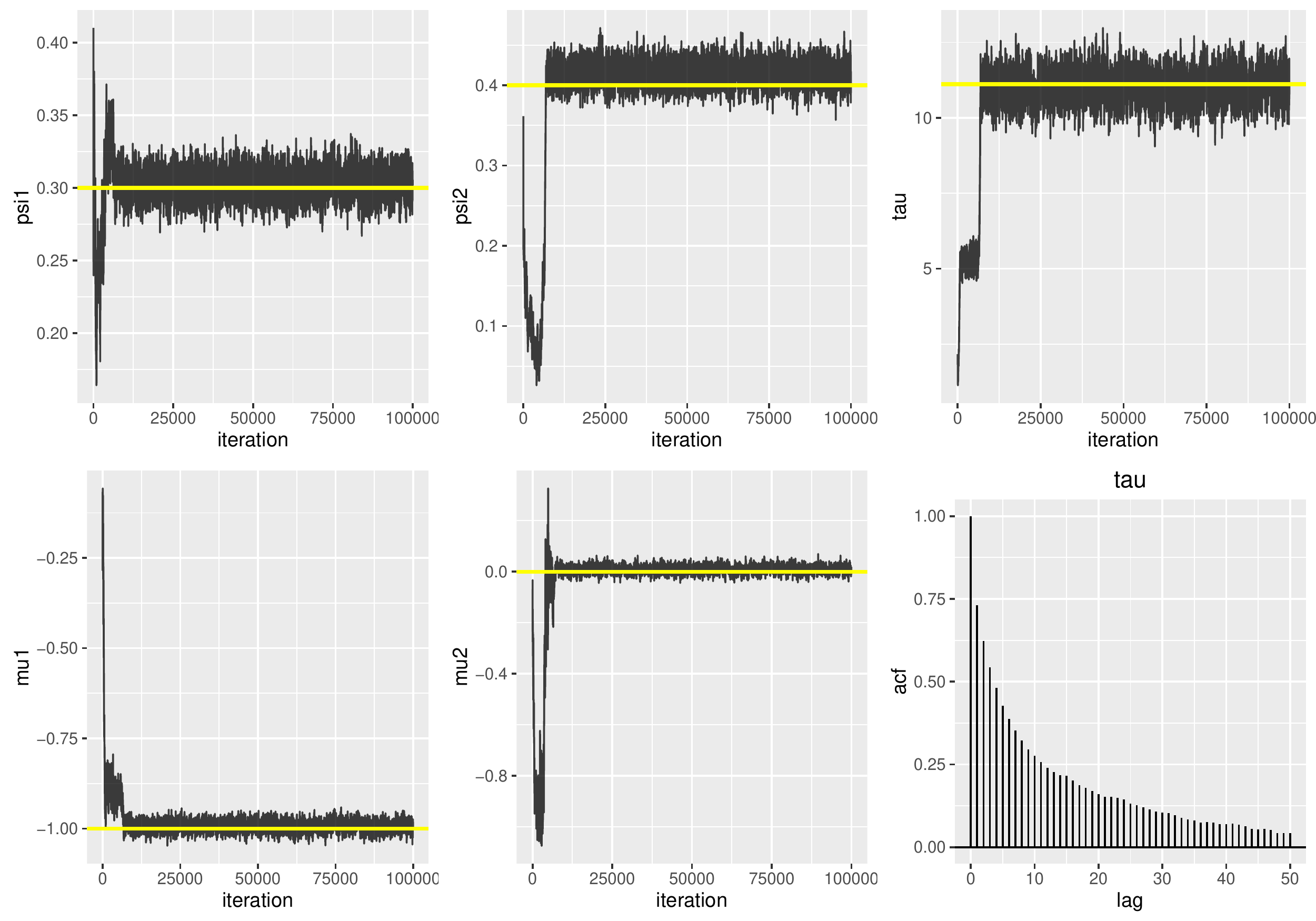}
\caption{Results for the example with a mixture of four normals using 100.000 MCMC iterations. 
The trace plots show all iterations; in the autocorrelation plot the first 20.000 iterations are treated as burnin. 
The figures are  obtained after subsampling the iterates, where only each 5th iterate was saved.
The horizontal yellow lines indicate true values. The results for the other parameters are similar and therefore not displayed.
}
\label{fig:multimodal-trace}
\end{figure}

\begin{figure}
\includegraphics[scale=0.6]{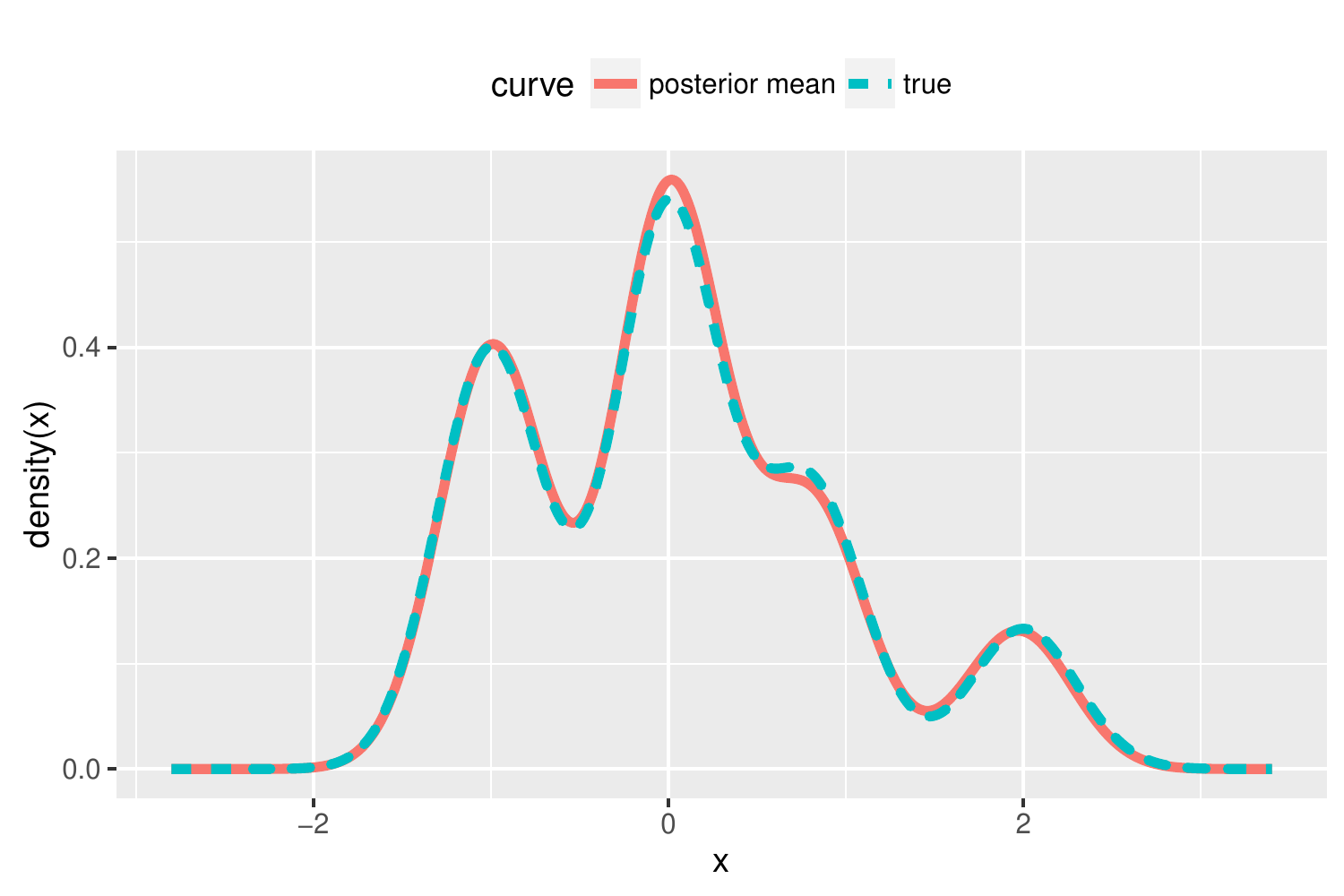}
\caption{Results for the example with a mixture of four normals;  the first 20.000 iterations are treated as burnin. Shown are the true jump size density and the density obtained from the posterior mean of the non-burnin iterates.
}\label{fig:multimodal-density}
\end{figure}

%----------

\begin{figure}
\includegraphics[scale=0.6]{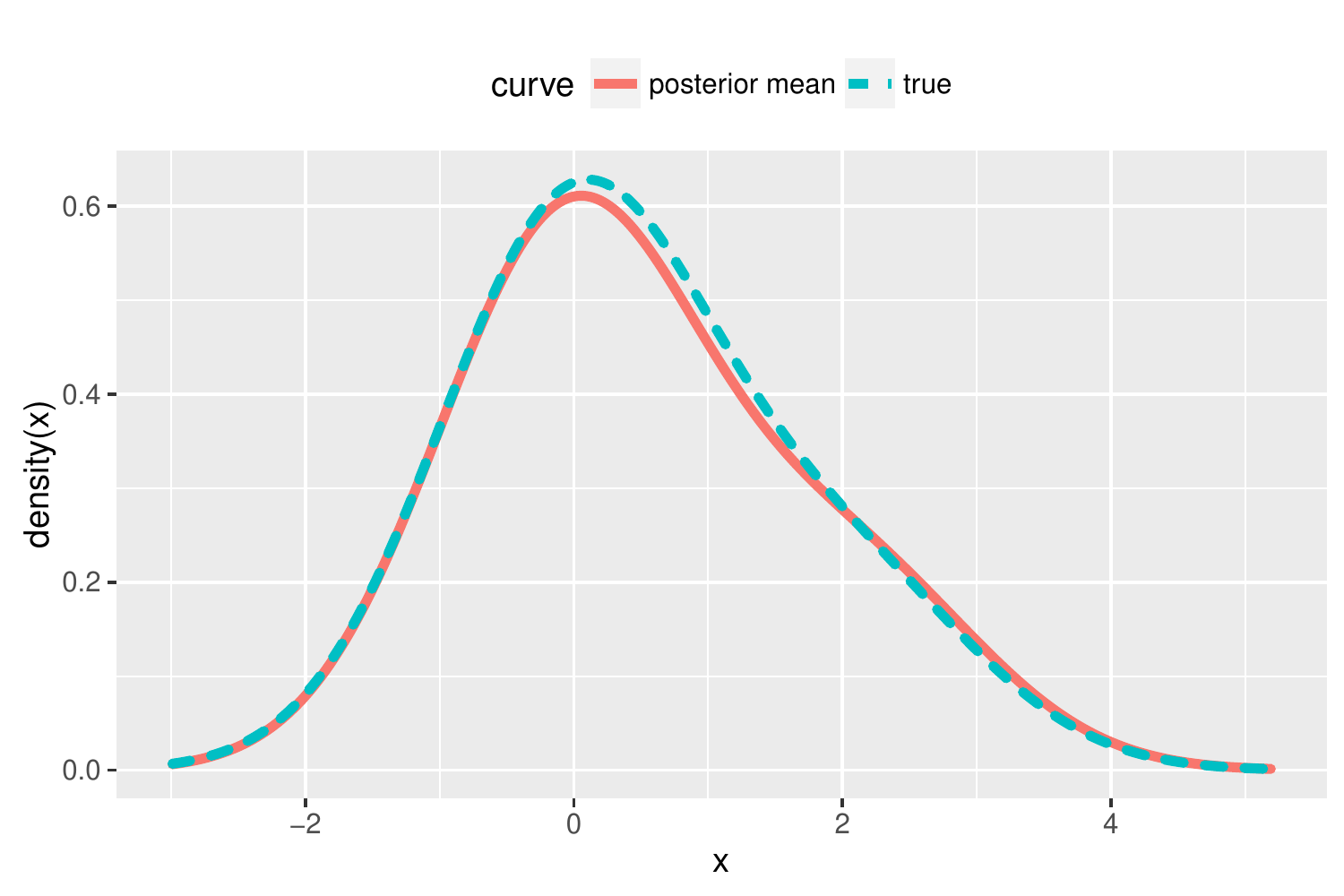}
\caption{Results for the example with a skew density;  the first 20.000 iterations are treated as burnin. Shown are the true jump size density and the density obtained from the posterior mean of the non-burnin iterates.
}\label{fig:skew-density}
\end{figure}

%%%%%%%%%%%%%%%%%%%%

\subsection{Discussion}
As can be seen from the autocorrelation plots, mixing of the chain deteriorates when $\la\Delta$ increases. As the focus in this article is on high frequency data, where there are on average only a few jumps in between observations, we do not go into details on improving the algorithm. We remark that a non-centred parametrisation (see for instance \cite{papas}) may give more satisfactory results when $\la\Delta$ is large.  A non centred parametrisation can be obtained by changing the hierarchical model in \eqref{eq:hierar2}. Denote by $F^{-1}_\la$ the inverse cumulative distribution function of the $\mathcal{P}(\la)$ distribution. Let $u_{ij}$ ($i=1,\ldots, n$ and $j=1,\ldots J$) be a sequence of independent $U(0,1)$ random variables and set $u=\{u_{ij},\, i=1,\ldots, n,\, j=1,\ldots J\}$. By considering the hierarchical model
\begin{align}\label{eq:hierar3}
 z_i \mid u, \mu, \tau  \quad \stackrel{\text{ind}}{\sim}  \quad &  N\left(\sum_{j=1}^J \mu_j F^{-1}_{\psi_j \Delta_i}(u_{ij}), \tau^{-1}\sum_{j=1}^J F^{-1}_{\psi_j \Delta_i}(u_{ij}) \right)\nonumber \\
u_{ij}\quad  \stackrel{\text{iid}}{\sim} \quad &  U(0,1) \\
(\psi, \mu, \tau)  \quad \sim \quad & \pi(\psi, \mu, \tau)\nonumber
\end{align}
($i\in \{1,\ldots, n\}$ and $j\in \{1,\ldots, J\}$), $\psi$ can be updated using  a Metropolis-Hastings step. In this way $\{n_{ij}\}$ and $\psi$ are updated simultaneously. 

Another option is to integrate out $(\mu,\tau)$ from $p(\th,z,\ba)$. In this model it is even possible to  integrate out $\psi$ as well. In that case only the auxiliary variables $\ba$ have to be updated. Yet another method to improve the efficiency of the algorithm is to use ideas from parallel tempering (cf.\ Chapter 11 in \cite{handbook-mcmc}).

%%%%%%%%%\mathrm{I}_n
\section{Proof of Theorem~\ref{mainthm}}
\label{sec:proofmainthm}

There are a number of general results in Bayesian nonparametric statistics, such as the fundamental Theorem 2.1 in \cite{ghosal00} and Theorem 2.1 in \cite{ghosal01}, which allow determination of the posterior contraction rates through checking certain conditions, but none of these results is easily and directly applicable in our case. The principle bottleneck is that a main assumption underlying these theorems is sampling from a fixed distribution, whereas in our high frequency setting, the distributions vary with $\Delta$. Therefore, for the clarity of exposition in the proof of our main theorem we will choose an alternative path, which consists in mimicking the main steps of the proof of Theorem~2.1, involving judiciously chosen statistical tests, as in \cite{ghosal00}, while also employing some results on the Dirichlet location mixtures of normal densities from \cite{ghosal01}. However, a significant part of technicalities we will encounter are characteristic of the decompounding problem only.

Throughout this section we assume that Assumptions~\ref{ass:truth} and~\ref{ass:prior} hold. Furthermore, in view of the discussion that followed Theorem~\ref{mainthm} we will without loss of generality assume that $0<\delta \le 4$. All the technical lemmas used in this section are collected in the appendices.

We start with the decomposition
\begin{equation}\label{eq:I+II}
\Pi( A(\varepsilon_n,M)|\mathcal{Z}_n^{\Delta} )=\Pi( A(\varepsilon_n,M)|\mathcal{Z}_n^{\Delta} )\phi_n+\Pi( A(\varepsilon_n,M)|\mathcal{Z}_n^{\Delta} )(1-\phi_n)=:\mathrm{I}_n+\mathrm{II}_n,
\end{equation}
where $0\leq\phi_n\leq 1$ is a sequence of tests based on observations $\mathcal{Z}_n^{\Delta}$ and with properties to be specified below. The idea is to show that the terms on the right-hand side of the above display separately converge to zero in probability. The tests $\phi_n$ allow one to control the behaviour of the likelihood ratio
\begin{equation*}
%\label{lr}
\mathcal{L}_n^{\Delta}(\lambda,f)=\prod_{i=1}^n \frac{k_{\lambda,f}^{\Delta}(Z_i^{\Delta})}{k_{\lambda_0,f_0}^{\Delta}(Z_i^{\Delta})}
\end{equation*}
on the set where it is not well-behaved due to the fact that $(\lambda,f)$ is `far away' from $(\lambda_0,f_0).$ 

%For bounding $\mathrm{I}_n$ and $\mathrm{II}_n$ we need to construct appropriate tests $\phi_n$. 

\subsection{Construction of tests}

The next lemma is an adaptation of Theorem 7.1 from \cite{ghosal00} to decompounding. A proof is given  in Appendix~\ref{appB}. We use the notation $D(\eps,A,d)$ to denote the $\eps$-packing number of a set $A$ in a metric space with metric $d$, applied in our case with $d$ the scaled Hellinger metric $h^\Delta$.
\begin{lemma}
\label{lemmA3}
Let $\mathcal{Q}$ be an arbitrary set of probability measures $\mathbb{Q}^\Delta_{\lambda,f}$.
Suppose for some non-increasing function $D(\eps),$ some sequence $\{\eps_n\}$ of positive numbers and every $\eps>\eps_n,$
\begin{equation}
\label{packn}
D\left( \frac{\eps}{2},\{ \mathbb{Q}^\Delta_{\lambda,f}\in\mathcal{Q}:\eps\leq h^\Delta(\mathbb{Q}^\Delta_{\lambda_0,f_0},\mathbb{Q}^\Delta_{\lambda,f})\leq 2\eps \},h^\Delta \right)\leq D(\eps).
\end{equation}
%with $D$ on the left-hand side denoting the $\eps/2$-packing number of the set
%\begin{equation*}
%\{ \mathbb{Q}^\Delta_{\lambda,f}\in\mathcal{Q}:\eps\leq h^\Delta(\mathbb{Q}^\Delta_{\lambda_0,f_0},\mathbb{Q}^\Delta_{\lambda,f})\leq 2\eps \}.
%\end{equation*}
Then for every $\eps>\eps_n$ there exists a sequence of tests $\{\phi_n\}$ (depending on $\eps>0$), such that
\begin{equation*}
\begin{split}
\ee_{\lambda_0,f_0}[\phi_n]&\leq D(\eps)\exp\left( -Kn\Delta\eps^2\right) \frac{1}{1-\exp\left(-Kn\Delta\eps^2\right)},%\label{A3.1}
\\
\sup_{\left\{\mathbb{Q}^\Delta_{\lambda,f}\in\mathcal{Q}:h^\Delta(\mathbb{Q}^\Delta_{\lambda_0,f_0},\mathbb{Q}^\Delta_{\lambda,f})>\eps\right\}}\ee_{\lambda,f}[1-\phi_n]&\leq \exp\left( -Kn\Delta\eps^2 \right),%, \label{A3.2}
\end{split}
\end{equation*}
where $K>0$ is a universal constant.
\end{lemma}
In the proofs of Propositions~\ref{prop:c1} and~\ref{prop:c3} 
we need the  inequalities below. There exists a constant $\overline{C}\in(0,\infty)$ depending on $\underline{\lambda}$ and $\overline{\lambda}$ only, such that for all $\lambda_1,\lambda_2\in [\underline{\lambda},\overline{\lambda}]$ and $f_1,f_2$ it holds that
\begin{align}
\mathrm{K}(\mathbb{Q}^{\Delta}_{\lambda_1,f_1},\mathbb{Q}^{\Delta}_{\lambda_2,f_2}) & \leq \overline{C}\Delta (\mathrm{K}(\mathbb{P}_{f_1},\mathbb{P}_{f_2})+|\lambda_1-\lambda_2|^2),\label{eq:1K}\\
\mathrm{V}(\mathbb{Q}^{\Delta}_{\lambda_1,f_1},\mathbb{Q}^{\Delta}_{\lambda_2,f_2}) & \leq \overline{C}\Delta(\mathrm{V}(\mathbb{P}_{f_1},\mathbb{P}_{f_2})+\mathrm{K}(\mathbb{P}_{f_1},\mathbb{P}_{f_2})+|\lambda_1-\lambda_2|^2), \label{eq:1V}\\
h(\mathbb{Q}^{\Delta}_{\lambda_1,f_1},\mathbb{Q}^{\Delta}_{\lambda_2,f_2})&\leq \overline{C}\sqrt{\Delta}(|\lambda_1-\lambda_2| +  
h(\mathbb{P}_{f_1},\mathbb{P}_{f_2})) \label{eq:1h}.
%\\
%&=\overline{h}_1(\lambda_0,\lambda)+\overline{h}_2(f_0,f),
\end{align}
%with an obvious definition of $\overline{h}_1,\overline{h}_2.$
These inequalities can be proven in the same way as  Lemma~1 in \cite{gugu15}.

Let $\eps_n$ be as in Theorem~\ref{mainthm}. Throughout,  $\overline{C}$ denotes the above constant. For a constant $L>0$ define the sequences $\{a_n\}$ and $\{\eta_n\}$ by 
\begin{equation*}
a_n=L\log^{2/\delta}\left( \frac{1}{\eta_n} \right), \quad \eta_n=\frac{\eps_n}{4\overline{C}},
\end{equation*}
We will show that inequality \eqref{packn} holds true for every $\eps=M\eps_n$ with $M>2$ and the set of measures $\mathcal{Q}$ equal to 
\begin{equation*}
\mathcal{Q}_n=\{ \bb{Q}_{\lambda,f_{H,\sigma}}^{\Delta}: \lambda \in [\underline{\lambda}, \overline{\lambda}], H[-a_n,a_n]\geq 1-\eta_n, \sigma \in [\underline{\sigma},\overline{\sigma}] \},
\end{equation*}
As a first step, note that we have
\begin{align}
\label{eq:DN}
\log D\left( \frac{\eps}{2},\mathcal{Q}_n,h^\Delta \right) & \leq \log D\left( {\eps_n},\mathcal{Q}_n,h^\Delta \right)\\
&\leq \log N\left( \frac{\eps_n}{2},\mathcal{Q}_n,h^\Delta \right)
= \log N\left( \frac{\eps_n\sqrt{\Delta}}{2},\mathcal{Q}_n,h \right), \nonumber
\end{align}
where $N\left(\frac{\eps_n\sqrt{\Delta}}{2},\mathcal{Q}_n,h\right)$ is the covering number of the set $\mathcal{Q}_n$ with $h$-balls of size $\eps_n\sqrt{\Delta}/2.$
The  first inequality in \eqref{eq:DN} follows from assuming $M>2$. 
For bounding the righthand side in \eqref{eq:DN}, we have the following proposition.
\begin{prop}
\label{prop:c1}
We have 
\begin{equation}
\label{c1v}
\log N\left(\frac{\eps_n\sqrt{\Delta}}{2},\mathcal{Q}_n,h\right) \lesssim \log^{4/\delta+1}\left( \frac{1}{\eps_n} \right),
\end{equation}
\end{prop}
\begin{proof}
Define
\begin{equation*}
 \scr{F}_n = 	\{ f_{H,\sigma}: H[-a_n,a_n]\ge 1-\eta_n, \sigma \in [\underline{\sigma}, \overline{\sigma}]\}. 	
\end{equation*}
Let $\{\lambda_i\}$ be centres of the balls from a minimal covering of $[\underline{\lambda}, \overline{\lambda}]$ with $|\cdot|$-balls of size $\eta_n$. Let $\{f_j\}$ be centres of the balls from a minimal covering of $\scr{F}_n$ with $h$-balls of size $\eta_n.$ For any $\bb{Q}_{\lambda, f_{H,\sigma}} \in \scr{Q }_n$, by \eqref{eq:1h} we have
\[ h(\bb{Q}_{\lambda, f_{H,\sigma}} , \bb{Q}_{\lambda_i, f_j}) \le  \frac{\eps_n\sqrt{\Delta}}{2}, \]
 by appropriate choices of $i$ and $j$. It follows that
\begin{equation*}
\log N \left(\frac{\eps_n\sqrt{\Delta}}{2},\mathcal{Q}_n,{h}\right) \leq \log N (\eta_n,[\underline{\lambda},\overline{\lambda}],|\cdot|) + \log N ( \eta_n ,  {\mathcal{F}}_n , {h} ).
\end{equation*}
Evidently,
\begin{equation*}
 \log N (\eta_n,[\underline{\lambda},\overline{\lambda}],|\cdot|) \lesssim \log \left( \frac{1}{\eps_n} \right).
\end{equation*}
As we assume $\delta \le 4$, we can apply the arguments on pp.~1251--1252 in \cite{ghosal01}, see in particular formulae (5.8)--(5.10) (cf.\ also Theorem 3.1 and Lemma A.3 there), which yield
\begin{equation*}
\log N ( \eta_n ,  {\mathcal{F}}_n , {h} ) \lesssim \log^{4/\delta+1}\left( \frac{1}{\eps_n} \right).
\end{equation*}
Combination of the above three inequalities implies the statement of the proposition.
\end{proof}
An application of Proposition~\ref{prop:c1} to \eqref{eq:DN} gives
\[	\log D\left( \frac{\eps}{2},\mathcal{Q}_n,h^\Delta \right)\lesssim 
\log^{4/\delta+1}\left( \frac{1}{\eps_n} \right) \le c_1 n\Delta\eps_n^2,
\]
for some positive constant $c_1$. Here,  the final inequality follows from our choice for $\eps_n$.
Hence,  \eqref{packn} is satisfied for 
\begin{equation*}
D(\eps)=\exp((c_1/M^2-K) n \Delta \eps^2).
\end{equation*}
By  Lemma~\ref{lemmA3} there exist tests $\phi_n$ such that for all $n$ large enough
\begin{align}
\ee_{\lambda_0,f_0}[\phi_n] &\leq 2 \exp\left( -(KM^2-c_1)n\Delta\eps_n^2 \right),\label{eq:tests0}
\\
\sup_{\left\{\mathbb{Q}^\Delta_{\lambda,f}\in\mathcal{Q}_n:h^\Delta(\mathbb{Q}^\Delta_{\lambda_0,f_0},\mathbb{Q}^\Delta_{\lambda,f})>\eps\right\}}\ee_{\lambda,f}[1-\phi_n]&\leq \exp\left( -Kn\Delta M^2\eps_n^2 \right)\label{eq:tests1}.
\end{align}

%-----------
\subsection{Bound on $\mathrm{I}_n$ in \eqref{eq:I+II}}

First note that by equation \eqref{eq:tests0}
\begin{equation*}
\ee_{\lambda_0,f_0}[ \mathrm{I}_n ] \leq \ee_{\lambda_0,f_0}[ \phi_n ] \le 2 \exp\left( -(KM^2-c_1)n\Delta\eps_n^2 \right).
\end{equation*}
Chebyshev's inequality implies that $\mathrm{I}_n$ converges to zero in $\mathbb{Q}_{\lambda_0,f_0}^{\Delta , n}$-probability as $n\rightarrow\infty,$ as soon as $M$ is chosen so large that $KM^2-c_1>0.$ \qed

\subsection{Bound on $\mathrm{II}_n$}

Now we consider $\mathrm{II}_n.$ We have
\begin{equation*}
\mathrm{II}_n=\frac{ \iint_{A(\varepsilon_n,M)} \mathcal{L}_n^{\Delta}(\lambda,f) \dd \Pi_1(\lambda) \mathrm{d}\Pi_2(f) (1-\phi_n) }{ \iint \mathcal{L}_n^{\Delta}(\lambda,f) \dd\Pi_1(\lambda) \mathrm{d}\Pi_2(f) }=:\frac{\mathrm{III}_n}{\mathrm{IV}_n}.
\end{equation*}
We will show that the numerator $\mathrm{III}_n$ goes exponentially fast to zero, in $\mathbb{Q}_{\lambda_0,f_0}^{\Delta , n}$-probability, while the denominator $\mathrm{IV}_n$ is bounded from below by an exponential function, with $\mathbb{Q}_{\lambda_0,f_0}^{\Delta , n}$-probability tending to one, in such a way that the ratio of $\mathrm{III}_n$ and $\mathrm{IV}_n$ still goes to zero in $\mathbb{Q}_{\lambda_0,f_0}^{\Delta , n}$-probability.
\medskip\\
{\it Bounding $\mathrm{III}_n$.} 
As $1_{\{A(\varepsilon_n,M)\}} \le 1_{\mathcal{Q}_n^c}+1_{\{A(\varepsilon_n,M)\cap \mathcal{Q}_n\}}$ we have
\begin{equation*}
\ee_{\lambda_0,f_0}[\mathrm{III}_n]
\leq \Pi(\mathcal{Q}_n^c)+\iint_{ \mathcal{Q}_n \cap A(\varepsilon_n,M) } \ee_{\lambda,f} [1-\phi_n]\dd\Pi_1(\lambda)\dd\Pi_2(f).
\end{equation*}
Here we applied Fubini's theorem to obtain the second term on the right-hand-side, which
by \eqref{eq:tests1}  is bounded by
$\exp(-KM^2n\Delta\eps_n^2).$
Furthermore,
\begin{equation*}
\Pi(\mathcal{Q}_n^c)=\Pi_2(H[-a_n,a_n]<1-\eta_n,\sigma\in[\underline{\sigma},\overline{\sigma}])\lesssim \frac{1}{\eta_n}e^{-ba_n^{\delta}},
\end{equation*}
%{\color{red} Please check: in the original version product measure $\Pi_1 \times \Pi_2$ appears. Begrijp ik niet.}
where the last inequality is formula (5.11) in \cite{ghosal01}. Hence
\begin{equation}
\label{iii.n}
\ee_{\lambda_0,f_0}[\mathrm{III}_n]\lesssim \frac{1}{\eta_n}e^{-ba_n^{\delta}} + \exp(-KM^2n\Delta\eps_n^2).
\end{equation}
\medskip\\
{\it Bounding $\mathrm{IV}_n$.} 
%\begin{equation}
%\label{bdel}
Recall $\mathrm{K}_\Delta=\mathrm{K}/\Delta$ and $\mathrm{V}_\Delta=\mathrm{V}/\Delta$. Let
\[	B^{\Delta}(\eps,(\lambda_0,f_0))=\left\{ (\lambda,f): K^\Delta( \mathbb{Q}^{\Delta}_{\lambda_0,f_0}, \mathbb{Q}^{\Delta}_{\lambda,f}) \leq\eps^2,  V^\Delta( \mathbb{Q}^{\Delta}_{\lambda_0,f_0}, \mathbb{Q}^{\Delta}_{\lambda,f}) \leq\eps^2 \right\}.\]
%\end{equation}
and
\begin{equation*}
\widetilde{\eps}_n=\frac{\log(n\Delta)}{\sqrt{n\Delta}}.
\end{equation*}
Note that $n\Delta \widetilde{\eps}_n^2\to\infty$ when $n\rightarrow\infty$. 

We will use the following bound, an adaptation of Lemma~8.1 in \cite{ghosal00} to our setting, valid for every $\eps>0$ and $C>0$,
\begin{equation}\label{eq:Qdelta}
\mathbb{Q}_{\lambda_0,f_0}^{\Delta , n}\big( \iint_{ B^{\Delta}(\eps,(\lambda_0,f_0)) } \mathcal{L}_n(\lambda,f)\dd\widetilde{\Pi}(\lambda,f) \leq \exp(-(1+C)n\Delta\eps^2)\big)
\leq \frac{1}{C^2n\Delta\eps^2},
\end{equation}
where
\begin{equation*}
\widetilde{\Pi}(\cdot)=\frac{\Pi(\cdot)}{\Pi( B^{\Delta}(\eps,(\lambda_0,f_0)) )}
\end{equation*}
is a normalised restriction of $\Pi(\cdot)$ to $B^{\Delta}(\eps,(\lambda_0,f_0))$.

By virtue of \eqref{eq:Qdelta}, with $\mathbb{Q}_{\lambda_0,f_0}^{\Delta , n}$-probability tending to one,  for any constant $C>0$ we have
\begin{align}
\lefteqn{\mathrm{IV}_n\geq\iint_{ B^{\Delta}(\widetilde{\eps}_n,(\lambda_0,f_0)) } \mathcal{L}^\Delta_n(\lambda,f)\dd \Pi_1(\lambda)\times\dd\Pi_2(f) } \label{eq:ivn} \\ & \qquad>\Pi( B^{\Delta}(\widetilde{\eps}_n,(\lambda_0,f_0)) )\,\,\exp(-(1+C)n\Delta \widetilde{\eps}_n^2).\nonumber
\end{align}
We will now work out the product probability on the right-hand side of this inequality.

\begin{prop}
\label{prop:c3}
It holds that
\begin{equation*}
\Pi\left(   B^{\Delta}(\widetilde{\varepsilon}_n,\mathbb{Q}_{\lambda_0,f_0})  \right) \gtrsim \exp\left( - \bar{c} \log^2 \left(\frac{1}{ \widetilde{\varepsilon}_n } \right) \right)
\end{equation*}
for some constant $\bar{c}.$
\end{prop}

\begin{proof}
Let $0<c\leq 1/\sqrt{5\overline{C}}$ be a constant. Here $\overline{C}$ is the constant in \eqref{eq:1K} and \eqref{eq:1V}. 
% and define the sets
%\begin{equation*}
%G(c\widetilde{\eps}_n,(\lambda_0,f_0))=\left\{ (\lambda,f):K(\pp_{f_0},\pp_{f}) \leq c^2 \widetilde{\eps}_n^2,V(\pp_{f_0},\pp_{f}) \leq c^2 \widetilde{\eps}_n^2 ,|\lambda_0-\lambda|^2 \leq c^2 \widetilde{\eps}_n^2\right\}.
%\end{equation*}
By these inequalities it is readily seen that
\[	\left\{ (\lambda,f):K(\pp_{f_0},\pp_{f}) \leq c^2 \widetilde{\eps}_n^2,V(\pp_{f_0},\pp_{f}) \leq c^2 \widetilde{\eps}_n^2 ,|\lambda_0-\lambda|^2 \leq c^2 \widetilde{\eps}_n^2\right\}\subset B^{\Delta}(\widetilde{\eps}_n, \bb{Q}_{\lambda_0, f_0}^{\Delta}).\]
%\begin{equation*}
%G(c\widetilde{\eps}_n,(\lambda_0,f_0)) \subset B^{\Delta}(\widetilde{\eps}_n, \bb{Q}_{\lambda_0, f_0}^{\Delta}).
%\end{equation*}
It then follows by the independence assumption on $\Pi_1$ and $\Pi_2$ that
\begin{align*}
%\label{pp}
%\begin{split}
\Pi(   B^{\Delta}(\widetilde{\varepsilon}_n,\mathbb{Q}^\Delta_{\lambda_0,f_0})  )
&\geq  \Pi_1\left( |\lambda_0-\lambda| \leq c \widetilde{\eps}_n \right) \\ &
 \times \Pi_2 \left( f : \mathrm{K}(\mathbb{P}_{f_0},\mathbb{P}_{f}) \leq c^2{\widetilde{\varepsilon}_n^2},  \mathrm{V}(\mathbb{P}_{f_0},\mathbb{P}_{f}) \leq c^2{\widetilde{\varepsilon}_n^2} \right).
%\end{split}
\end{align*}
For the first factor on the right-hand side we have by \eqref{pi1} that
\begin{equation*}
\Pi_1\left( |\lambda_0-\lambda| \leq  c \widetilde{\eps}_n \right) \gtrsim \widetilde{\eps}_n.
\end{equation*}
As far as the second factor is concerned, for some constants $\overline{c}_1,\overline{c}_2$ it is bounded from below by
\[\overline{c}_1 \exp\left( - \overline{c}_2 \log^2 \left(\frac{1}{ \widetilde{\eps}_n } \right) \right),\]
by the same arguments as in  inequality (5.17) in \cite{ghosal01}. The result now follows by combining the two lower bounds.
\end{proof}
Combining \eqref{eq:ivn} with Proposition~\ref{prop:c3}, with $\mathbb{Q}_{\lambda_0,f_0}^{\Delta, n}$-probability tending to one as $n\rightarrow\infty,$ for any constant $C>0$ we have
\begin{align}
\mathrm{IV}_n &> \exp\left(-(1+C)n\Delta \widetilde{\eps}_n^2-\bar{c} \log^2 \left(\frac{1}{ \widetilde{\varepsilon}_n } \right)\right).\label{eq:ivn2}
\end{align}
We are now ready for showing the final steps of proving that $\mathrm{II}_n$ tends to zero in $\mathbb{Q}_{\lambda_0,f_0}^{\Delta, n}$-probability. Let $G_n$ denote the set on which Inequality~\eqref{eq:ivn2} is true. Then by~\eqref{iii.n} we obtain
\begin{align*}
\ee_{\lambda_0,f_0}[ \mathrm{II}_n 1_{G_n} ] & \lesssim \exp\left((1+C)n\Delta \widetilde{\eps}_n^2+\bar{c} \log^2 \left(\frac{1}{ \widetilde{\varepsilon}_n } \right)\right) \\
& \qquad\times \left[ \frac{1}{\eta_n}e^{-ba_n^{\delta}} + \exp(-KM^2n\Delta\eps_n^2) \right].
\end{align*}
Recall that $n\Delta \widetilde{\eps}_n^2=\log^2(n\Delta)$. Hence, the exponent in the first factor of this display is of order $\log^2 (n\Delta)$. Furthermore $a_n^\delta=L^\delta\log^2({4\overline{C}}/{\eps_n})$, which is of order $\log^2(n\Delta)$ as well. 
It follows that, provided the constants $L$ and $M$ are chosen large enough, the right-hand side of the above display converges to zero as $n\rightarrow\infty.$ Chebyshev's inequality then implies that $\mathrm{II}_n$ converges to zero in probability as $n\rightarrow\infty.$ This completes the proof of Theorem~\ref{mainthm}. \qed

\bigskip

{\bf Acknowledgement}: We wish to thank   Wikash Sewlal from Delft University of Technology for the simulation results of the  example with a mixture of four normals and the skewed density.

\appendix

\section{Additional lemmas and proofs}

\subsection{Proof of Lemma~\ref{lemma:delta}}\label{app:proof}

We give a detailed proof of Equality~\eqref{eq:hhdelta}. As we are interested in small values of $\Delta$, we make some necessary approximations. Starting point is the expansion for the `density' of $\mathbb{Q}^\Delta_{\lambda,f}$ with respect to the Lebesgue measure,
\[
e^{-\lambda\Delta}\delta_0(x)+(1-e^{-\lambda\Delta})\sum_{m=1}^{\infty} a_m(\lambda\Delta) f^{\ast m}(x),
\]
see \eqref{density}, with coefficients $a_m$ defined in \eqref{am}. It follows that we have the likelihood ratio
\begin{align*}
\frac{\dd \mathbb{Q}^\Delta_{\lambda,f}}{\dd \mathbb{Q}^\Delta_{\lambda_0,f_0}}(x) & =
\one_{x=0}e^{-(\lambda-\lambda_0)\Delta}+\one_{x\neq 0}\frac{(1-e^{-\lambda\Delta})\sum_{m=1}^{\infty} a_m(\lambda\Delta) f^{\ast m}(x)}{(1-e^{-\lambda_0\Delta})\sum_{m=1}^{\infty} a_m(\lambda_0\Delta) f_0^{\ast m}(x)} \\
& = e^{-(\lambda-\lambda_0)\Delta}\left(\one_{x=0}+\one_{x\neq 0}\frac{\lambda f(x)}{\lambda_0 f_0(x)}+o(\Delta)\right),
\end{align*}
where we collected terms of order $\Delta^m$ for $m\geq 2$ as $o(\Delta)$. Hence we get for the Hellinger affinity
\[
H(\mathbb{Q}^\Delta_{\lambda,f},\mathbb{Q}^\Delta_{\lambda_0,f_0}) = \int \sqrt{ \dd \mathbb{Q}^\Delta_{\lambda,f}   \dd \mathbb{Q}^\Delta_{\lambda_0,f_0} }
\]
the approximating expression 
\[
H(\mathbb{Q}^\Delta_{\lambda,f},\mathbb{Q}^\Delta_{\lambda_0,f_0})= e^{-(\lambda+\lambda_0)\Delta/2}\left(1+\Delta\sqrt{\lambda_0\lambda}H(f,f_0)+o(\Delta)\right).
\]
It follows that for $\Delta\to 0$, 
\begin{align*}
h^2(\mathbb{Q}^\Delta_{\lambda,f},\mathbb{Q}^\Delta_{\lambda_0,f_0})& = 2-2H(\mathbb{Q}^\Delta_{\lambda,f},\mathbb{Q}^\Delta_{\lambda_0,f_0}) \\
& = 2-2e^{-(\lambda+\lambda_0)\Delta/2}\left(1+\Delta\sqrt{\lambda_0\lambda}H(f,f_0)+o(\Delta)\right)\\
& = 2(1-e^{-(\lambda+\lambda_0)\Delta/2})-2e^{-(\lambda+\lambda_0))\Delta/2}\left(\Delta\sqrt{\lambda_0\lambda}H(f,f_0)+o(\Delta)\right).
\end{align*}
Hence, for $\Delta\to 0$,
\begin{align*}
\frac{1}{\Delta}h^2(\mathbb{Q}^\Delta_{\lambda,f},\mathbb{Q}^\Delta_{\lambda_0,f_0})& \to \lambda+\lambda_0-2\sqrt{\lambda_0\lambda}H(f,f_0) \\
& =  \int(\sqrt{\lambda f(x)}-\sqrt{\lambda_0f_0(x)})^2\,\dd x.
\end{align*}
Equality~\eqref{eq:hhdelta} follows. The proofs of the equalities \eqref{eq:kkdelta} and \eqref{eq:vvdelta} follow a similar line  of reasoning.
%-----------------

\subsection{Proof of Lemma~\ref{lemmA3}}\label{appB}

The proof is an adaptation of Theorem 7.1 from \cite{ghosal00} to decompounding.
In all what follows it is assumed that $\mathbb{Q}^\Delta_{\lambda,f}\in\mathcal{Q}$, but we suppress this assumption in the notation. Observe that
\begin{multline*}
D\left( \frac{\eps}{2},\{ \mathbb{Q}^\Delta_{\lambda,f}:\eps\leq h^\Delta(\mathbb{Q}^\Delta_{\lambda_0,f_0},\mathbb{Q}^\Delta_{\lambda,f})\leq 2\eps \},h^\Delta \right)\\
=D\left( \frac{\eps\sqrt{\Delta}}{2},\{ \mathbb{Q}^\Delta_{\lambda,f}:\eps\sqrt{\Delta}\leq h(\mathbb{Q}^\Delta_{\lambda_0,f_0},\mathbb{Q}^\Delta_{\lambda,f})\leq 2\eps\sqrt{\Delta} \},h \right).
\end{multline*}
From this point on the arguments from the proof of Theorem 7.1 in \cite{ghosal00} are applicable (with $\eps$ replaced by $\eps\sqrt{\Delta}$) and eventually lead to the desired result. The role of formulae (7.1)--(7.2) in that proof are played in the present context by \eqref{eq:t1} and \eqref{eq:t2} below.

For a given $(\lambda_1,f_1)$ there exists a sequence of tests $\phi_n$ based on $\mathcal{Z}_n^{\Delta},$ such that
\begin{align}
\ee_{\lambda_0,f_0}[\phi_n] & \leq \exp\left( -\frac{1}{2}n\Delta h^\Delta(\mathbb{Q}^\Delta_{\lambda_0,f_0},\mathbb{Q}^\Delta_{\lambda,f})^2 \right),\label{eq:t1}\\
\sup_{ h^\Delta(\mathbb{Q}^\Delta_{\lambda,f},\mathbb{Q}^\Delta_{\lambda_1,f_1}) < h^\Delta(\mathbb{Q}^\Delta_{\lambda_0,f_0},\mathbb{Q}^\Delta_{\lambda_1,f_1}) } \ee_{\lambda,f}[1-\phi_n]&\leq \exp\left( -\frac{1}{2}n\Delta h^\Delta(\mathbb{Q}_{\lambda_0,f_0},\mathbb{Q}_{\lambda,f})^2 \right).\label{eq:t2}
\end{align}

These two inequalities simply follow by rewriting  the inequalities
\begin{align*}
\ee_{\lambda_0,f_0}[\phi_n] & \leq \exp\left( -\frac{1}{2}nh^2(\mathbb{Q}_{\lambda_0,f_0}^{\Delta},\mathbb{Q}_{\lambda,f}^{\Delta}) \right),\\
\sup_{ h(\mathbb{Q}_{\lambda,f}^{\Delta},\mathbb{Q}_{\lambda_1,f_1}^{\Delta}) < h(\mathbb{Q}_{\lambda_0,f_0}^{\Delta},\mathbb{Q}_{\lambda_1,f_1}^{\Delta}) } \ee_{\lambda,f}[1-\phi_n]&\leq \exp\left( -\frac{1}{2}n h^2(\mathbb{Q}_{\lambda_0,f_0}^{\Delta},\mathbb{Q}_{\lambda,f}^{\Delta}) \right).
\end{align*}
which are proved on pp.~520--521 in \cite{ghosal00} and rely upon the results in \cite{birge84} and \cite{lecam86}.

%-----------------
%\subsection{Additional lemmas}
%\label{app}

\subsection{Proof of Lemma~\ref{lem:postpars}}\label{proof2}

As the priors for $\psi_1,\ldots, \psi_J$  are independent, we obtain that 
\begin{align*} p(\psi \mid \mu, \tau, z, \ba) &= p(\psi \mid \ba) \propto \prod_{j=1}^J \left( e^{-\psi_j T} 
\psi_j^{s_j} \pi(\psi_j)\right)\\ &=\prod_{j=1}^J \left( e^{-(\psi_j T+\beta_0)} 
 \psi_j^{s_j+\alpha_0-1} \right), 
\end{align*}
 which proves the first statement of the lemma.

For $(\mu, \tau)$ we get 
\begin{align*}
p(\mu, \tau \mid z, \ba) &\propto 
\prod_{i \in \scr{I}} \phi\left(z_i ; a_i'\mu, n_i/\tau\right) 
\\ & \times \tau^{\alpha_1-1} e^{-\beta_1\tau} \tau^{J/2} \exp\left(-\frac{\tau\kappa}{2} \sum_{j=1}^J (\mu_j-\xi_j)^2\right).
\end{align*}
This is proportional to 
\[ \tau^{\alpha_1-1+(I+J)/2}  \exp\left(-\beta_1\tau - \frac{D(\mu)}{2} \tau \right), \]
where
\[ D(\mu)=\kappa \sum_{j=1}^J (\mu_j-\xi_j)^2 + \sum_{i \in \scr{I}} n_i^{-1} \left( z_i - a_i' \mu\right)^2. \]
From this expression it is easily seen that  we can integrate out $\mu$ to obtain the distribution of $\tau$, conditional on $(z, \ba)$.  To get this right,  write $D(\mu)$ as a quadratic form of $\mu$: 
\[ D(\mu) = \mu' P \mu - 2q' \mu +R.  \]
By completing the square, we find that 
\[  \int \exp\left( -\frac{\tau}{2} D(\mu) \right) d \mu = e^{-\tau R/2} \int \exp\left( -\frac12 \mu \tau P \mu + \tau q' \mu \right) d \mu.\]
The integrand is (up to a proportionality constant), the density of a bivariate normal random vector with mean vector $P^{-1}q$ and covariance matrix $\tau^{-1} P^{-1}$ evaluated in $\mu$. This implies that the preceding display equals
\[	e^{-\tau R/2}  (2 \pi)^{J/2} \sqrt{|\tau^{-1} P^{-1}|} \exp\left( \frac12 \tau q' P^{-1} q\right). \]
We conclude that 
\[	p(\tau \mid z, \ba)  \propto 
\tau^{\alpha_1+I/2-1} \exp\left(-(\beta_1+\frac12(R - q' P^{-1} q))\tau\right),\]
which proves the asserted Gamma distribution of $\tau$. This computation also immediately leads to the assertion on the distribution of $\mu$. We finally show that the rate parameter appearing for $\tau$ is positive.
By definition $D(\mu)\ge 0$ for all $\mu$. This implies that  
$D(P^{-1} q)=q' P^{-1} q-2 q' P^{-1}q+R=R-q' P^{-1} q\ge 0$.

\bibliographystyle{plainnat}

\begin{thebibliography}{99}

\bibitem[Alexandersson(1985)]{Alexandersson1985} H.~Alexandersson. A Simple Stochastic Model of the Precipitation Process. \emph{Journal of Climate and Applied Meteorology} 24(12): 1282--1295, 1985.

\bibitem[L\'evy Matters IV(2015)]{levymatters} D.~Belomestny, F.~Comte, V.~Genon-Catalot, H.~Masuki, M.~Rei{\ss} (eds.).\ \emph{L\'evy matters IV, Estimation for discretely observed L\'evy processes}.\ Lecture Notes in Mathematics 2128. Springer, Cham, 2015.

\bibitem[Birg\'e(1984)]{birge84} L.\ Birg\'e. Sur un th\'eor\`eme de minimax et son application aux tests.\ \emph{Probab.\ Math.\ Statist.}, 3:259--282, 1984. 

\bibitem[Brooks et al.(2011)]{handbook-mcmc} S.~Brooks, A.~Gelman, G.L.~Jones and X.L.~Meng. \emph{Handbook of Markov Chain Monte Carlo}, Chapman \& Hall/CRC, 2011.

%\bibitem[Brockett et al.(2013)]{brockett13} P.L.\ Brockett, W.N.\ Hudson and H.G.\ Tucker.\ The distribution of the likelihood ratio for additive processes.\ \emph{J.\ Multivariate Anal.}, 8:233--243, 1978.
\bibitem[Buchmann and Gr\"ubel(2003)]{buchmann03} B.\ Buchmann and R.\ Gr\"ubel.\
Decompounding: an estimation problem for Poisson random sums.\ \emph{Ann.\ Statist.}, 31:1054--1074, 2003.
\bibitem[Buchmann and Gr\"ubel(2004)]{buchmann04} B.\ Buchmann and R.\ Gr\"ubel.\ Decompounding Poisson random sums: recursively truncated estimates in the discrete case.\ 
\emph{Ann.\ Inst.\ Statist.\ Math.}, 56:743--756, 2004. 

\bibitem[Burlando and Rosso(1993)]{br1993} P.~Burlando and R.~Rosso, Stochastic Models of Temporal Rainfall: Reproducibility, Estimation and Prediction of Extreme Events. In: \emph{Stochastic Hydrology and its Use in Water Resources Systems, Simulation and Optimization}, J.B. Marco, R. Harboe, J.D. Salas (eds.), NATO ASI Series 237: 137--173. Springer, 1993. 

\bibitem[Comte and Genon-Catalot(2009)]{cgc2009} F.~Comte and V.~Genon-Catalot. Nonparametric estimation for pure jump L\'evy processes based on high frequency data. \emph{Stochastic Processes and their Applications} 119(12): 4088--4123, 2009.

\bibitem[Comte and Genon-Catalot(2010a)]{cgc2010a} F.~Comte and V.~Genon-Catalot. Non-parametric estimation for pure jump irregularly sampled or noisy L\'evy processes. \emph{Statistica Neerlandica} 64(3): 290--313, 2010.

\bibitem[Comte and Genon-Catalot(2010b)]{cgc2010b} F.~Comte and V.~Genon-Catalot. Nonparametric adaptive estimation for pure jump L\'evy processes. \emph{Annales de Institut Henri Poincare (B), Probability and Statistics} 46(3): 595--617, 2010.


\bibitem[Comte and Genon-Catalot(2015)]{cgc} F.~Comte and V.~Genon-Catalot. Adaptive estimation for {L}\'evy processes. In: D.~Belomestny, F.~Comte, V.~Genon-Catalot, H.~Masuki, M.~Rei{\ss} (eds.).\ \emph{L\'evy matters IV, Estimation for discretely observed L\'evy processes}.\ Lecture Notes in Mathematics 2128: 77--177. Springer, Cham, 2015.

\bibitem[Comte et al.(2014)]{comte13} F.\ Comte, C. Duval and V.\ Genon-Catalot.\ Nonparametric density estimation in compound Poisson process using convolution power estimators.\ \emph{Metrika}, 77:163--183, 2014.
\bibitem[Comte and Genon-Catalot(2011)]{comte11} F.\ Comte and V.\ Genon-Catalot. Estimation for L\'evy processes from high frequency data within a long time interval. \emph{Ann.\ Statist.}, 39:803--837, 2011.
%\bibitem[Csisz\'ar~(1963)]{csiszar1963}I.~Csisz{\'a}r.\ Eine informationstheoretische {U}ngleichung und ihre {A}nwendung auf den {B}eweis der {E}rgodizit\"at von {M}arkoffschen {K}etten, \emph{Magyar Tud. Akad. Mat. Kutat\'o Int. K\"ozl.}, 8:85--108, 1963.
%\bibitem[Csisz\'ar and Shields(2004)]{csiszar04} I.\ Csisz\'ar and P.C.\ Shields.\ Information theory and statistics: a tutorial.\ \emph{Found.\ Trends Commun.\ Inf.\ Theory}, 1:417--528, 2004.
\bibitem[Diebolt and Robert(1995)]{diebolt-robert} 
J.\ Diebolt, and C. P.\ Robert. {\it Estimation of finite mixture distributions through Bayesian
sampling}, J. Roy. Statist. Soc. Ser. B 56: 363--375,  1994.
\bibitem[Duval(2013)]{duval12} C.\ Duval.\ Density estimation for compound Poisson processes from discrete data.\ \emph{Stoch.\ Proc.\ Appl.}, 123:3963--3986, 2013.
%\bibitem[Diaconis and Freedman(1986)]{diaconis86} P.\ Diaconis and D.\ Freedman. On the consistency of Bayes estimates. With a discussion and a rejoinder by the authors. \emph{Ann. Statist.}, 14:1--67, 1986.
\bibitem[Embrechts et al.(1997)]{embrechts97} P.\ Embrechts, C.\ Kl\"uppelberg and T.\ Mikosch. Modelling Extremal Events for Insurance and Finance.\ Applications of Mathematics (New York), 33. Springer-Verlag, Berlin, 1997.
\bibitem[van Es et al.(2007)]{vanes07} B.\ van Es, S.\ Gugushvili and P.\ Spreij.\ A kernel type nonparametric density estimator for decompounding.\ \emph{Bernoulli}, 13:672--694, 2007.
%\bibitem[Fan(1991)]{fan91} J.\ Fan.\ {On the optimal rates of convergence for nonparametric deconvolution problems}. \emph{Ann. Statist.}, 19:1257--1272, 1991.
\bibitem[Ferguson(1973)]{ferguson73} T.S. Ferguson.\ A Bayesian analysis of some nonparametric problems.\ \emph{Ann.\ Statist.}, 1:209--230, 1973.
\bibitem[Ferguson(1983)]{ferguson83} T.S. Ferguson.\ Bayesian density estimation by mixtures of normal distributions.\ \emph{Recent Advances in Statistics}, 287--302, Academic Press, New York, 1983.

\bibitem[Figueroa-L\'opez(2008)]{fl2008} J.E.~Figueroa-L\'opez. Small-time moment asymptotics for L\'evy processes. \emph{Statistics
and Probability Letters} 78(18): 3355--3365, 2008.

\bibitem[Figueroa-L\'opez(2009)]{fl2009} J.E.~Figueroa-L\'opez. Nonparametric estimation of L\'evy models based on discrete-
sampling. In: \emph{Optimality. IMS lecture notes monograph series} 57: 117--146,
Beachwood, OH: Institute of Statistical Mathematics, 2009.

%\bibitem[Fr\"uhwirth-Schnatter (1998)]{fruh}  Fr\"uhwirth-Schnatter, S. (2001)
% {\it Markov chain Monte Carlo estimation of classical and dynamic
%switching and mixture models}, J. Amer. Statist. Assoc. {\bf 96}, 194--209.
%\bibitem[Fr\"uhwirth-Schnatter (2006)]{fruh-book}  Fr\"uhwirth-Schnatter, S. (2006) {\it Finite Mixture and Markov Switching Models}, Springer.
%\bibitem[Gelman et al (2013)]{gelman} Gelman, A., Carlin, J.B., Stern, H.S., Dunson, D.B., Vehtari, A. and Rubin, D.B. (2013) {\it Bayesian Data Analysis}, Third Edition, Chapman \& Hall/CRC Texts in Statistical Science). 

\bibitem[Ghosal(2010)]{ghosal10} S.\ Ghosal.\ The Dirichlet process, related priors and posterior asymptotics. \emph{Bayesian nonparametrics}, 35--79, Camb.\ Ser.\ Stat.\ Probab.\ Math., Cambridge Univ.\ Press, Cambridge, 2010.
\bibitem[Ghosal et al.(2000)]{ghosal00} S.\ Ghosal, J.K.\ Ghosh and A.W.\ van der Vaart. Convergence rates of posterior distributions. \emph{Ann.\ Statist.}, 28:500--531, 2000.
\bibitem[Ghosal and Tang(2006)]{ghosal06} S.\ Ghosal and Y.\ Tang.\ Bayesian consistency for Markov processes.\ \emph{Sankhy\=a}, 68:227--239, 2006. 
\bibitem[Ghosal and van der Vaart(2001)]{ghosal01} S.\ Ghosal and A.W.\ van der Vaart. Entropies and rates of convergence for maximum likelihood and Bayes estimation for mixtures of normal densities.
\emph{Ann.\ Statist.}, 29:1233--1263, 2001. 
%\bibitem[Ghosal and van der Vaart(2007)]{ghosal07} S.\ Ghosal and A.W.\ van der Vaart. Convergence rates of posterior distributions for non-i.i.d.\ observations. \emph{Ann.\ Statist.}, 35:192--223, 2007.
\bibitem[Ghosal and van der Vaart(2007)]{vdv07} S.\ Ghosal and A.W.\ van der Vaart.\ Posterior convergence rates of Dirichlet mixtures at smooth densities. \emph{Ann.\ Statist.}, 35:697--723, 2007.
\bibitem[Gugushvili et al.(2015)]{gugu15} S.\ Gugushvili, F.\ van der Meulen and P.\ Spreij.\ Non-parametric Bayesian inference for multi-dimensional compound Poisson processes.\ \emph{Modern Stochastics: Theory and Applications}, 2:1--15, 2015.
\bibitem[Hjort et al.(2010)]{hjort10} N.L.\ Hjort, C.\ Holmes, P.\ M\"uller and S.G. Walker. \emph{Bayesian Nonparametrics.} Cambridge Series in Statistical and Probabilistic Mathematics, 28. Cambridge University Press, Cambridge, 2010.
%\bibitem[Hobert et al. (2011)]{hobert} Hobert, J.P.,  Roy, V. and Robert, C.P. (2011) {\it Improving the Convergence Properties of the Data Augmentation Algorithm with an Application to Bayesian Mixture Modeling} Statist. Sci.
%{\bf 26}(3), 332-351.
%\bibitem[Hobert and Marchev (2008)]{hobert-marchev} Hobert, J.P. and Marchev, D. (2008) {\it A theoretical comparison of the data augmentation, marginal augmentation and PX-DA algorithms} 
%Ann. Statist. {\bf 36}(2),  532-554.
\bibitem[Ibragimov and Khas'minski{\u\i}(1982)]{ibragimov82} I.A.\ Ibragimov and R.Z. Khas'minski{\u\i}.\ An estimate of the density of a distribution belonging to a class of entire functions (Russian). \emph{Teor.\ Veroyatnost.\ i Primenen.}, 27:514--524, 1982.
\bibitem[Insua et al.(2012)]{insua12} D.R.~Insua, F.~Ruggeri and M.P.~Wiper.\ \emph{Bayesian Analysis of Stochastic Process Models}. John Wiley \& Sons, 2012.
\bibitem[Jacod and Shiryaev(2003)]{js} J.~Jacod and A.N.~Shiryaev. \emph{Limit theorems for stochastic processes, Second edition}. Grundlehren der Mathematischen Wissenschaften, 288. Springer-Verlag, Berlin, 2003.

\bibitem[Katz(2002)]{katz} R.W.~Katz. Stochastic modeling of hurricane damage. \emph{The Journal of Applied Meteorology}, 41(7): 754--762, 2002.


\bibitem[Kutoyants(1998)]{kutoyants98} Yu.A.\ Kutoyants. \emph{Statistical Inference for Spatial Poisson Processes.} Lecture Notes in Statistics, 134. Springer-Verlag, New York, 1998.
\bibitem[Le Cam(1986)]{lecam86} L.\ M.\ Le Cam.\ \emph{Asymptotic Methods in Statistical Decision Theory}. Springer, New York, 1986.
\bibitem[Lo(1984)]{lo84} A.Y.\ Lo.\ On a class of Bayesian nonparametric estimates: I. Density estimates. \emph{Ann.\ Statist.}, 12:351--357, 1984.

\bibitem[McLachlan and Peel(2000)]{mp2000}
G.~McLachlan and D.~Peel. \emph{Finite mixture models}. Wiley Series in Probability and Statistics: Applied Probability and Statistics. Wiley-Interscience, New York, 2000.

\bibitem[Marron and Wand(1992)]{mw1992} J.S.~Marron and M.P.~Wand.
Exact Mean Integrated Squared Error.
\emph{Ann. Statist.}
20(2): 712--736, 1992.


\bibitem[Papaspiliopoulos et al.(2007)]{papas}O.\ Papaspiliopoulos, G.O.\ Roberts and M.\ Sk\"old. \emph{A General Framework for the Parametrization of Hierarchical Models}. Statistical Science, 22(1), 59--73

%\bibitem[Pollard(2002)]{pollard2002} D.~Pollard. \emph{A User's Guide to Measure Theoretic Probability}. Cambridge Series in Statistical and Probabilistic Mathematics, 8. Cambridge University Press, Cambridge, 2002.
\bibitem[Prabhu(1998)]{prabhu98} N.U.\ Prabhu.\ \emph{Stochastic Storage Processes. Queues, Insurance Risk, Dams, and Data Communication.} Second edition. Applications of Mathematics (New York), 15. Springer-Verlag, New York, 1998.
\bibitem[Richardsen and Green(1997)]{richardsen-green} S.\ Richardsen  and  P.J.\ Green.  
{\it On Bayesian Analysis of Mixtures with an Unknown Number of Components (with discussion)} 
Journal of the Royal Statistical Society: Series B (Statistical Methodology), 59: 731--792, 1997.
%\bibitem[Roberts and Stramer (2001)]{roberts-stramer}  Roberts, G.O.\ and  Stramer, O.  (2001) {\it On inference
%for partially observed nonlinear diffusion models using
%the Metropolis–Hastings algorithm} Biometrika {\bf 88}, 603–621.

\bibitem[Scalas(2006)]{scalas}  E.~Scalas. The application of continuous time random walks in finance and economics. \emph{Physica A}, 362(2): 225--239, 2006.

\bibitem[Shreve(2008)]{shreve} S.E.\ Shreve, (2008) {\it Stochastic Calculus for Finance II}, 2nd edition, Springer.
%\bibitem[Rasmussen and Williams(2006)]{rasmussen06} C.E.\ Rasmussen and C.K.I.\ Williams.\ \emph{Gaussian Processes for Machine Learning.} Adaptive Computation and Machine Learning. MIT Press, Cambridge, MA, 2006.
%\bibitem[Shiryaev(1996)]{shiryaev1996}
%A.N.~Shiryaev. \emph{Probability}.
%Graduate Texts in Mathematics, 95,
%Springer-Verlag, New York, 1996.
\bibitem[Skorohod(1964)]{skorohod64} A.V.\ Skorohod.\ \emph{{\cyr{\emph{Slucha{\u i}nye protsessy s nezavisimymi prirashcheniyami.}}} (Russian) [Random Processes with Independent Increments]}. Izdat.\ ``Nauka'', Moscow, 1964.
\bibitem[Tang and Ghosal(2007)]{tang07} Y.\ Tang and S.\ Ghosal.\ Posterior consistency of Dirichlet mixtures for estimating a transition density. \emph{J.\ Statist.\ Plann.\ Inference}, 137:1711--1726, 2007. 
%\bibitem[Tanner and Wong(1987)]{tanner-wong} M.A.\ Tanner and W.H.\ Wong. {\it
%The Calculation of Posterior Distributions by Data Augmentation}
% Journal of the American Statistical Association {\bf  82}(398), pp 528--540,  1987.
 
\bibitem[Tanner and Wong(1987)]{tanner-wong} M.A.\ Tanner and W.H.\ Wong. 
The Calculation of Posterior Distributions by Data Augmentation.
 {\it Journal of the American Statistical Association} 82:528--540,  1987.

\bibitem[Ueltzh\"ofer and Kl\"uppelberg(2011)]{uk2011} F. A. J.~Ueltzh\"ofer and C.~Kl\"uppelberg. An oracle inequality for penalised projection estimation of L\'evy densities from high frequency observations. \emph{Journal of Nonparametric Statistics} 23(4): 967--989, 2011.

%\bibitem[Vajda(1972)]{vajda72} I.\ Vajda.\ On the $f$-divergence and singularity of probability measures. Collection of articles dedicated to the memory of Alfr\'ed R\'enyi, I.\ \emph{Period.\ Math.\ Hungar.}, 2:223--234, 1972.
\bibitem[van der Vaart and Wellner(2000)]{wellner00} A.W.\ van der Vaart and J.A.\ Wellner.\ \emph{Weak Convergence and Empirical Processes. With Applications to Statistics}. Corrected second printing.\ Springer Series in Statistics.\ Springer-Verlag, New York, 2000.
%\bibitem[van der Vaart and van Zanten(2008a)]{vdv08a}  A.W.\ van der Vaart and J.H.\ van Zanten. Rates of contraction of posterior distributions based on Gaussian process priors. \emph{Ann.\ Statist.}, 36:1435--1463, 2008a.
%\bibitem[van der Vaart and van Zanten(2008b)]{vdv08b} A.W.\ van der Vaart and J.H.\ van Zanten. Reproducing kernel Hilbert spaces of Gaussian priors. \emph{Pushing the Limits of Contemporary Statistics: Contributions in Honor of Jayanta K. Ghosh}, 200--222, Inst. Math. Stat. Collect., 3. Inst. Math. Statist., Beachwood, OH, 2008b.
%\bibitem[Wasserman(1998)]{wasserman98} L.\ Wasserman. Asymptotic properties of nonparametric Bayesian procedures. \emph{Practical Nonparametric and Semiparametric Bayesian Statistics}, 293--304, Lecture Notes in Statist., 133, Springer, New York, 1998.






\end{thebibliography}

\end{document}